\documentclass{art}
\usepackage{mypreamble}
\usepackage[normalem]{ulem} 

\begin{document}
\begin{frontmatter}

  \title{Convex Team Logics}
  \runtitle{Convex Team Logics}

  \author{\fnms{Aleksi}
    \snm{Anttila}
    \ead[label=e1]{aleksi.anttila@unipd.it}%
  }%
  \address{Department of Philosophy, Sociology, Education and Applied Psychology\\
    University of Padua\\
    Palazzo del Capitanio, Piazza Capitaniato 3 \\
    35139 Padua\\
    ITALY\\
    \href{mailto:aleksi.anttila@unipd.it}{aleksi.anttila@unipd.it}\\
  }%
  \and%
  \author{\fnms{Søren Brinck}
    \snm{Knudstorp}
    \ead[label=e2]{s.b.knudstorp@uva.nl}
}
    \corref{}
  \address{Institute for Logic, Language and Computation \& Philosophy\\
    University of Amsterdam\\
    Science Park 107\\
    1098 XG Amsterdam\\
    NETHERLANDS\\
     \href{mailto:s.b.knudstorp@uva.nl}{s.b.knudstorp@uva.nl}\\
  }%
  \runauthor{A.~Anttila and S.~B.~Knudstorp}

\begin{abstract}
We prove expressive completeness results for \emph{convex} propositional and modal team logics, where a logic is convex if, for each formula, if it is true in two teams $t$ and $u$ and $t\subseteq s\subseteq u$, then it is also true in $s$. We introduce propositional/modal team logics which are expressively complete for the class of all convex propositional/modal team properties. We also answer an open question concerning the expressive power of classical propositional logic extended with the \emph{nonemptiness atom} $\NE$: we show that it is expressively complete for the class of all convex and \emph{union-closed} properties. A modal analogue of this result additionally yields an expressive completeness theorem for Aloni's (2022) Bilateral State-Based Modal Logic. 

There is a specific sense in which one of the propositional convex logics we introduce extends \emph{propositional dependence logic}; this sense corresponds to the existence of a \emph{compositional translation} between the two logics. Another of the convex logics extends \emph{propositional inquisitive logic} in the same manner. We generalize the notion of \emph{uniform definability}, as considered in the team semantics literature, to formalize this notion of \emph{compositional extension}.
\end{abstract}

\begin{keyword}[class=AMS]
  \kwd[Primary ]{03B60} \kwd[ Secondary ]{03B45} \kwd{03B65}
\end{keyword}

\begin{keyword}
\kwd{team semantics} \kwd{convexity} \kwd{uniform definability} \kwd{bilateral state-based modal logic} \kwd{dependence logic} \kwd{inquisitive logic} \kwd{translations}
\end{keyword}

\end{frontmatter}

\nocite{andersonbelnapdunn92}
\nocite{urquhart72}
\nocite{fine17}
\nocite{fraassen69}
\nocite{knudstorp22}
\nocite{knudstorp23}
\nocite{vanbenthem96}
\nocite{bezhetal24}
\nocite{dosen89}
\nocite{tritschler2025}
\nocite{abramsky}
\nocite{burgess}
\nocite{kontinen2011}

\section{Introduction} \label{c:section:introduction}

\emph{Team semantics} provides a formal framework for reasoning about phenomena involving pluralities of data by representing them semantically as sets of evaluation points, called \emph{teams}. Formulas are then interpreted over such teams, as opposed to over single evaluation points. Thus, in predicate team semantics, formulas are evaluated relative to sets of variable assignments \cite{vaananen2007}; in propositional team semantics, sets of valuations  \cite{yangvaananen2016}; and in modal team semantics, sets of possible worlds \cite{vaananen2008}. This idea was first introduced by Hodges \cite{hodges1997,hodges1997b} in his \emph{trump semantics} for Hintikka and Sandu's \emph{independence-friendly logic} \cite{hintikka1989, hintikka1996}; team semantics in its contemporary form was subsequently introduced by Väänänen in his work on \emph{dependence logic} \cite{vaananen2007}. \emph{Inquisitive logic}, developed independently by Ciardelli, Groenendijk, and Roelofsen \cite{CiardelliRoelofsen2011,inqsembook,ciardellibook}, also essentially employs team semantics---see \citep{yang2014,ciardelli2016,ciardellibook}.

Team-semantic \emph{closure properties} such as \emph{downward closure} (a formula $\phi$ is downward closed just in case its truth in a team implies truth in all subteams: $[t\vDash \phi$ and $s\subseteq t]\implies s\vDash \phi$) and \emph{union closure} ($\phi$ is union closed iff given a nonempty collection of teams $T$, if $t\vDash \phi$ for all $t\in T$, then  $\bigcup T\vDash \phi$) play an important role in the study of team logics. They allow for concise and tractable characterizations of these logics, and provide an effective tool for proving definability results, completeness of axiomatizations (see, e.g., \cite{yangvaananen2016,yang2017}), and uniform interpolation \cite{dagostino}. 

The closure properties can also be conceptually suggestive. For example, on the common interpretation of team logics in which teams represent \emph{information states}, a formula $\phi$ being downward closed can be understood as representing that the information, or content, expressed by $\phi$ is \emph{persistent}: if $\phi$ is established in an information state $t$ ($t\vDash \phi$), then moving to a more informed state $s$ by ruling out possibilities ($s\subseteq t$)  preserves the information that $\phi$ ($s\vDash \phi$) (see, e.g., \cite{veltman,inqsembook,yalcin,bledin,aloni2022}). (Compare the sentence `It is raining' with the sentence `It might be raining' featuring the \emph{epistemic modality} `might'. The former is persistent: if I know for a fact that it is raining at present, then there is no further information that could invalidate `It is raining'. `It might be raining', on the other hand, is not persistent: if, for all I know, it might be raining, and I learn that it is, in fact, not raining, the information or content expressed by `It might be raining' is invalidated by the further information.)

In this paper, we introduce the \emph{convexity} closure property: a formula $\phi$ is convex if its truth in teams $s$ and $t$ implies its truth in all teams $u$ with $s\subseteq u \subseteq t$ (see Figure \ref{c:figure:convexity}). Intuitively, $\phi$ is convex if there are no ``gaps'' in its denotation $\left \Vert \phi\right \Vert=\{t\mid t\vDash \phi\}$---the meaning of $\phi$ is continuous (convexity is sometimes referred to as \emph{continuity}, as in \cite{vanbenthem1984}). Convexity can also be seen as a generalization of downward closure: if a formula is downward closed, it is also convex, though the converse need not hold. Many natural properties, or denotations or meanings, which are not expressible in a downward-closed setting become expressible in a convex setting; for instance, the epistemic modality-example from above, while not expressible by any downward-closed formula, corresponds to a simple convex formula.

  \begin{figure}[t]  
\centering
\begin{subfigure}[t]{0.3\textwidth}
\centering
\begin{tikzpicture}[scale=0.58]
\tikzset{every node/.style={fill=none,draw=black,circle,inner sep=2pt}}
\tikzset{every edge/.style={draw,-latex}}
\tikzset{every loop/.style={min distance=10mm,in=0,out=60,looseness=10}}

\node[fill,draw=black,circle,inner sep=1pt] (N) at (0,-4) {};
\node[fill,draw=black,circle,inner sep=1pt] (xyz) at (0,2) {};

\node[fill,draw=black,circle,inner sep=1pt] (x) at (-2,-2) {};
\node[fill,draw=black,circle,inner sep=1pt] (y) at (0,-2) {};
\node[fill,draw=black,circle,inner sep=1pt] (z) at (2,-2) {};
\node[draw=none,inner sep=0pt] (a) at (0.4,-4.4) {};

\node[fill,draw=black,circle,inner sep=1pt] (xy) at (-2,0) {};
\node[fill,draw=black,circle,inner sep=1pt] (xz) at (0,0) {};
\node[fill,draw=black,circle,inner sep=1pt] (yz) at (2,0) {};

\draw (N) edge (x);
\draw (N) edge (y);
\draw (N) edge (z);

\draw (x) edge (xy);
\draw (x) edge (xz);
\draw (y) edge (xy);
\draw (y) edge (yz);
\draw (z) edge (xz);
\draw (z) edge (yz);

\draw (xy) edge (xyz);
\draw (xz) edge (xyz);
\draw (yz) edge (xyz);

  \draw[fill=blue, fill opacity=0.1,rounded corners] plot[] coordinates {([shift={(-15pt,-23pt)}]x) ([shift={(-15pt,15pt)}]xy) ([shift={(15pt,15pt)}]yz) ([shift={(15pt,-30pt)}]z) ([shift={(0pt,-15pt)}]xz)  ([shift={(-15pt,-30pt)}]x) ([shift={(-15pt,-23pt)}]x)};

\end{tikzpicture}
\caption{Convex}
\end{subfigure}
\begin{subfigure}[t]{0.3\textwidth}
\centering
\begin{tikzpicture}[scale=0.58]
\tikzset{every node/.style={fill=none,draw=black,circle,inner sep=2pt}}
\tikzset{every edge/.style={draw,-latex}}
\tikzset{every loop/.style={min distance=10mm,in=0,out=60,looseness=10}}

\node[fill,draw=black,circle,inner sep=1pt] (N) at (0,-4) {};
\node[fill,draw=black,circle,inner sep=1pt] (xyz) at (0,2) {};

\node[fill,draw=black,circle,inner sep=1pt] (x) at (-2,-2) {};
\node[fill,draw=black,circle,inner sep=1pt] (y) at (0,-2) {};
\node[fill,draw=black,circle,inner sep=1pt] (z) at (2,-2) {};
\node[draw=none,inner sep=0pt] (a) at (0.4,-4.4) {};

\node[fill,draw=black,circle,inner sep=1pt] (xy) at (-2,0) {};
\node[fill,draw=black,circle,inner sep=1pt] (xz) at (0,0) {};
\node[fill,draw=black,circle,inner sep=1pt] (yz) at (2,0) {};

\draw (N) edge (x);
\draw (N) edge (y);
\draw (N) edge (z);

\draw (x) edge (xy);
\draw (x) edge (xz);
\draw (y) edge (xy);
\draw (y) edge (yz);
\draw (z) edge (xz);
\draw (z) edge (yz);

\draw (xy) edge (xyz);
\draw (xz) edge (xyz);
\draw (yz) edge (xyz);

 \draw[fill=blue, fill opacity=0.1,rounded corners] plot[] coordinates {([shift={(-5pt,-10pt)}]y) ([shift={(-15pt,0pt)}]xy) ([shift={(0pt,15pt)}]xyz) ([shift={(15pt,0pt)}]yz) ([shift={(0pt,-15pt)}]y) ([shift={(-5pt,-10pt)}]y)};

\end{tikzpicture}
\caption{Convex (and upward-directed)}
\end{subfigure}
\begin{subfigure}[t]{0.3\textwidth}
\centering
\begin{tikzpicture}[scale=0.58]
\tikzset{every node/.style={fill=none,draw=black,circle,inner sep=2pt}}
\tikzset{every edge/.style={draw,-latex}}
\tikzset{every loop/.style={min distance=10mm,in=0,out=60,looseness=10}}

\node[fill,draw=black,circle,inner sep=1pt] (N) at (0,-4) {};
\node[fill,draw=black,circle,inner sep=1pt] (xyz) at (0,2) {};

\node[fill,draw=black,circle,inner sep=1pt] (x) at (-2,-2) {};
\node[fill,draw=black,circle,inner sep=1pt] (y) at (0,-2) {};
\node[fill,draw=black,circle,inner sep=1pt] (z) at (2,-2) {};
\node[draw=none,inner sep=0pt] (a) at (0.4,-4.4) {};

\node[fill,draw=black,circle,inner sep=1pt] (xy) at (-2,0) {};
\node[fill,draw=black,circle,inner sep=1pt] (xz) at (0,0) {};
\node[fill,draw=black,circle,inner sep=1pt] (yz) at (2,0) {};

\draw (N) edge (x);
\draw (N) edge (y);
\draw (N) edge (z);

\draw (x) edge (xy);
\draw (x) edge (xz);
\draw (y) edge (xy);
\draw (y) edge (yz);
\draw (z) edge (xz);
\draw (z) edge (yz);

\draw (xy) edge (xyz);
\draw (xz) edge (xyz);
\draw (yz) edge (xyz);

\draw[fill=red,rounded corners,fill opacity=0.1]  (0,-0.5) -- (-2.5,-0.5) -- (-2.5,0) --  (0,2.5) -- (2.5,0) -- (2.5,-0.5) -- (0,-0.5);

 \draw[fill=red, fill opacity=0.1, rounded corners=3pt]
  ([shift={(-13pt,-13pt)}]N) rectangle
  ([shift={(13pt,13pt)}]N);

\end{tikzpicture}
\caption{Not convex (but upward-directed)}
\end{subfigure}
\caption{Examples of convex/non-convex subsets of a lattice. Note that in a powerset lattice, upward-directedness corresponds to union closure.}
\label{c:figure:convexity}
\end{figure}

 The convexity closure property is the team-semantic instantiation of a more general notion of gaplessness in meanings. There have been  a number of proposals to the effect that such an absence of gaps is a common, or even essential, feature of  meaning---that convexity constitutes a linguistic or cognitive universal of some kind. Examples  include Barwise and Cooper's monotonicity constraint for simple noun phrases \cite{barwise1981}; van Benthem's expectation that ``reasonable'' quantifiers exhibit convexity \cite{vanbenthem1984}; and Gärdenfors' proposal that concepts conform to convexity \cite{gardenfors2000}.

However, despite convexity being prominent in the broader literature, the corresponding closure property has not been studied within team semantics (although see the recent~\cite{hella2024,ciardelli2024,häggblom2026expressibilityinexpressibilitypropositionalteam} for other uses of convexity, and \cite{,anttila2024,anttila2026complexity,barberoyang2026}, which build on the results in the present paper). Here, we aim to fill this gap by proving \emph{expressive completeness} results for propositional and modal team logics with respect to convex meanings, i.e. w.r.t. classes of convex properties.

In the first part of the paper, we introduce propositional team logics which we prove to be expressively complete with respect to the class of all convex propositional properties (where a propositional property is a class of propositional teams). That is, for each of these logics, we show that (1) its formulas are convex, and conversely (2) any convex class of teams can be expressed by some formula of the logic.

We also introduce modal extensions of these logics, and prove that a modal analogue of the expressive completeness result holds for one of the extensions.

A typical (if oft-implicit) team-logical desideratum is to use teams to produce more expressive yet \emph{conservative} extensions of existing logics with standard Tarskian semantics. The logics to be introduced meet this requirement by conservatively extending, respectively, classical propositional logic $\mathrm{CPL}$ and the least normal modal logic $\mathrm{K}$. This follows by a standard team-semantical argument: they each include a fragment that is \emph{functionally complete} for $\mathrm{CPL}$ (resp. $\mathrm{K}$) and, such that, for each formula $\alpha$ of the fragment, 
$$t\vDash \alpha \iff v\vDash \alpha \text{ for all }v\in t.$$
That is, $\alpha$ is true (according to the team semantics of the logic in question) in a team $t$ just in case it is true in all elements of the team (according to the standard semantics of the fragment). We call such a fragment the \emph{classical basis} of the team logics that extend it.

The standard classical basis in the lineage of Väänänen's dependence logic includes the \emph{split disjunction} $\vee$ (also known as the \emph{tensor disjunction} or the \emph{local disjunction}). We show that in a setting in which all convex properties are definable, $\vee$ does not preserve convexity---that is, one can find convex formulas $\phi$ and $\psi$ such that $\phi\vee\psi$ is not convex. To obtain logics that are complete for all convex properties, we must therefore break with this lineage and either modify the split disjunction or opt for a different classical basis; we pursue both approaches. 

On the first approach, we introduce \emph{convex propositional dependence logic} $\CONDEP$. $\CONDEP$ is essentially \emph{propositional dependence logic} $\DEP$ (expressively complete for all downward-closed properties with the \emph{empty team property}) \cite{yangvaananen2016} extended with an operator $\filleddiamond$ corresponding to the epistemic `might' discussed above, except with the split disjunction $\vee$ of $\DEP$ replaced with a variant $\veedot$, called the \textit{ideal disjunction}, designed to force downward closure and hence convexity.

On the second approach, we start instead from a different prominent team logic, namely \emph{propositional inquisitive logic} $\INQ$ \cite{CiardelliRoelofsen2011,ciardellibook}. Like propositional dependence logic, it is  expressively complete for the class of all downward-closed properties with the empty team property, making it another candidate for a team logic with a natural convex extension. The classical basis of $\INQ$ (unlike that of $\DEP$) does preserve convexity, and we show that $\PLIM$, i.e. the extension of this basis with the $\filleddiamond$-operator, is expressively complete for all convex properties. Inquisitive logic, $\INQ$, on the other hand, also includes the \emph{inquisitive disjunction} $\vvee$ (a.k.a. the \emph{global disjunction}) which, like the split disjunction, does not preserve convexity. We define a variant $\vveedot$ of $\vvee$, called the \emph{downjunction} that, similarly to the ideal disjunction $\veedot$ forces convexity by forcing downward closure. Extending $\PLIM$ with $\vveedot$ gives a third logic complete for the class of all convex properties, \emph{convex inquisitive logic} $\CONINQ$. Although $\CONINQ$ cannot express more properties than $\PLIM$, there is a different sense in which the addition of $\vveedot$ does increase expressivity---we discuss this below.

In the setting of modal team semantics, we show, as with $\vee$ and $\vvee$, that the standard diamond modality $\diamonddiamond$ 
employed in the dependence logic lineage \cite{vaananen2008,hella2015,kontinen20142} fails to preserve convexity. To obtain a modal team logic complete for the class of all (team bisimulation-invariant) convex modal properties, we extend $\PLIM$ with the \emph{flat modalities} $\Diamond$ and $\Box$ employed in some formulations of modal inquisitive logic \cite{ciardelli2016} and, more recently, in Aloni's \emph{Bilateral State-Based Modal Logic} ($\BSML$) \cite{aloni2022}. We also show that the modal extensions of $\CONDEP$ and $\CONINQ$ are not bisimulation-invariant.

In the second part of the paper, we focus on a subclass of convex properties: convex union-closed properties. (Union-closed logics such as \emph{inclusion logic} \cite{galliani2012,hella2015,yang2022} are another central family of team logics.) We show that when restricted to a union-closed setting, the split disjunction $\vee$ \emph{does} preserve convexity. And so the usual dependence logic classical basis can be extended to a logic expressively complete for all convex union-closed properties. Indeed, we answer a problem left open in \cite{yang2017} by showing that $\PLNE$, the extension of this basis with the \emph{nonemptiness atom} $\NE$ \cite{vaananen2014,yang2014,yang2017}---true in a team just in case the team is nonempty---is complete for all convex and union-closed properties. A modal analogue of this yields an expressive completeness result for $\BSML$, which answers another open problem \cite{aloni2023} (one that, in fact, motivated this line of research). As with $\vee$, the modality $\diamonddiamond$ does preserve convexity in a union-closed setting, and so we also establish a modal analogue of the convex, union-closed expressive completeness result for $\PLNE$ with the modalities $\diamonddiamond$ and $\boxbox$ (a universal modality commonly paired with $\diamonddiamond$).

In the third and final part of the paper, we pursue a different line of inquiry. Using the aforementioned as impetus, we develop new team-semantical notions of definability, extension and translation.

A motivating example comes from the fact that $\vee$ and $\vvee$ are not definable in our convex logics (since $\vee,\vvee$ do not preserve convexity); in established terminology \cite{kontinen2009,ciardelli2009,galliani2013b,yang2014,yang20173,ciardelli2019,hella2024,barbero2024}, they are not \emph{uniformly definable} in the logics. Nevertheless, there is a restricted but, we shall argue, interesting sense in which all connectives of dependence logic $\DEP$ \emph{are} definable in convex dependence logic $\CONDEP$, and an accompanying sense in which the latter logic may be said to extend the former (and likewise for the inquisitive logics $\INQ$ and $\CONINQ$). We define a generalization of uniform definability and use it to capture this notion of extension, called \emph{compositional extension}, as well as several other related extension notions. As a closing result, we justify this choice of terminology (and, we believe, the novel definability and extension notions) by proving that a logic $\LOGIC_2$ is a compositional extension of a logic $\LOGIC_1$ if and only if there is a \emph{compositional translation} from the latter, $\LOGIC_1$, to the former, $\LOGIC_2$. 

The paper is structured as follows. In Section \ref{c:section:convex}, we work in the convex setting. We define the propositional logics $\CONDEP$, $\CONINQ$, and $\PLIM$, and show that each of them is expressively complete with respect to the class of all convex propositional properties. We then show that the extension of $\PLIM$ with the flat modalities $\Diamond$ and $\Box$ is expressively complete for the class of all convex modal properties invariant under bounded team bisimulation. We further show that the disjunctions $\vee$ and $\vvee$ and the modality $\diamonddiamond$ fail to preserve convexity, and that the modal extensions we introduce featuring either $\veedot$ or $\vveedot$ are not invariant under team bisimulation. In Section \ref{c:section:union-closed_convex}, we turn to the convex and union-closed setting. We prove that the logic $\PLNE$ is expressively complete with respect to the class of all convex and union-closed propositional properties, and that two distinct modal extensions of $\PLNE$, as well as $\BSML$, are expressively complete for the class of all convex and union-closed modal properties invariant under bounded bisimulation. In Section \ref{c:section:uniform_definability}, we generalize uniform definability to define the notion of compositional extension as well as other extension notions, and show how compositional extensions are related to compositional translations. Section \ref{c:section:conclusion} concludes and lists some problems for further research.
\section{Convex Properties}
\label{c:section:convex}

In Section \ref{c:section:convex_prop}, we introduce three propositional logics which we show to be expressively complete with respect to the class of all convex properties. Thereafter, in Section \ref{c:section:convex_modal}, we introduce modal extensions of these logics and prove a modal analogue of the expressive completeness results.

\subsection{Propositional Properties}
\label{c:section:convex_prop}

We define the syntax and semantics of the propositional logics of concern; recall basic facts about propositional team semantics and team-semantic closure properties; show that the split disjunction $\vee$ and the inquisitive disjunction $\vvee$ fail to preserve convexity; and prove expressive completeness results for the class of all convex propositional properties.

\subsubsection{Preliminaries: Logics and (Classical) Languages} \label{c:section:convex_prop_preliminaries} 

There are many sets of connectives \emph{functionally complete} for classical propositional logic (e.g. $\{\land, \lnot\}$): each Boolean function is definable using the connectives in such a set. One can define team-based versions of these sets of connectives to obtain different versions of team-based classical propositional logic---different \emph{classical bases} of propositional team logics. Whereas, due to the functional completeness of the different alternatives, the choice of connectives (often) makes no difference when one is studying the properties of classical propositional logic, the choice of connectives in the classical basis of a team logic is significant because, as we will see, it may affect the expressive power of its nonclassical extensions.

The two most prominent propositional team logics---\emph{propositional dependence logic} $\DEP$\cite{yangvaananen2016} and \emph{propositional inquisitive logic} $\INQ$ \cite{CiardelliRoelofsen2011,ciardellibook}---adopt distinct classical bases. The classical basis of the former includes the standard team-based disjunction connective, \textit{split disjunction} $\vee$; and it extends this classical basis with non-classical \emph{dependence atoms} $\dep{p_1,\ldots}{p_n}$, which are used to represent functional dependence relations between the values taken by propositional letters in a team. In contrast, the classical basis of inquisitive logic does not feature the split disjunction, but instead includes the \emph{intuitionistic implication} $\to$. It extends this basis with the non-classical \emph{inquisitive/global disjunction} $\vvee$, which is used to model the meanings of questions.

We will define what we call \emph{convex variants} of these logics: \textit{convex dependence logic} $\CONDEP$ and \textit{convex inquisitive logic} $\CONINQ$. The former, $\CONDEP$, extends classical propositional logic with both dependence atoms and the \emph{epistemic might}-operator $\filleddiamond$, used to model epistemic modalities such as the `might' in `It might be raining'. The classical basis of $\CONDEP$ is slightly different from that of $\DEP$: it replaces the split disjunction $\vee$ with a variant, denoted $\veedot$ and called the \emph{ideal disjunction}. $\CONINQ$ shares its classical basis with $\INQ$, but extends this basis with both the operator $\filleddiamond$ and a variant $\vveedot$ of $\vvee$, called the \emph{downjunction}. We also introduce a third convex logic, denoted $\PLIM$; this is the $\vveedot$-free fragment of $\CONINQ$. See Table \ref{c:table:convex_logics}.

\begin{table}[h]
    \centering
\begin{tabular}{c|c||c|c}
    Downward-closed logic & Classical basis & Convex variant/logic & Classical basis \\
    \hline
    Dependence logic:  & $\PLV$ & Convex dependence logic:  & $\PLD$\\
    $\DEP$ & & $\CONDEP$ &\\
    \hline
    Inquisitive logic: & $\PLI$ & Convex inquisitive logic:  & $\PLI$\\
    $\INQ$ & & $\CONINQ$ &\\
    \hline
     &  & $\PLIM$ & $\PLI$
\end{tabular}
    \caption{The logics and their classical bases.
    }
    \label{c:table:convex_logics}
\end{table}

Formally, we define the relevant logics, languages and semantic notions as follows. (Our presentation of dependence atoms deviates from the usual presentation in the literature: instead of treating an $n$-ary atom $\dep{p_1,\ldots}{p_n}$ as a primitive symbol, we consider it to be a compound formula yielded by combining the propositional letters $p_1,\ldots,p_n$ with an \emph{$n$-ary (functional) dependence connective $=_n\hspace{-0.1cm}(\cdot)$.} This difference in presentation is immaterial for Sections \ref{c:section:convex}--\ref{c:section:union-closed_convex}; we discuss it in Section \ref{c:section:uniform_definability}.)

\begin{definition}[Syntax]
    \label{c:def:syntax_PLNE} 
    Fix a countably infinite set of propositional letters $\mathsf{P}$. The formulas of \emph{classical propositional logic (with $\vee$/with $\veedot$/with $\to$)} $\PLV$/$\PLD$/$\PLI$ are given by the BNF-grammars
    \begin{align*}
        &\alpha\mathrel{::=} p \mid \lnot p \mid  \bot \mid \top \mid \alpha\land \alpha \mid \alpha\lor\alpha\tag{$\PLV$}\\
            &\alpha\mathrel{::=} p \mid \lnot p \mid \bot \mid \top \mid  \alpha\land \alpha \mid \alpha\dor\alpha \tag{$\PLD$} \\
    &\alpha\mathrel{::=} p  \mid \bot \mid \alpha\land \alpha \mid \alpha\to\alpha \tag{$\PLI$}
    \end{align*}
    where $p\in \mathsf{P}$.
    
    The formulas of propositional dependence logic $\PLV(\con{\cdot})$, convex propositional dependence logic $\CONDEP$, propositional inquisitive logic $\INQ$, convex propositional inquisitive logic $\CONINQ$, and $\PLIM$ are, respectively, given by the BNF-grammars 
    \begin{align*}
    &\phi\mathrel{::=} p \mid \lnot p \mid  \bot \mid \top \mid  \phi\land \phi \mid \phi\vee\phi  \mid {=_n\hspace{-0.1cm}(p_1,\ldots,p_{n})} \tag{$\DEP$}\\
   &\phi\mathrel{::=} p \mid \lnot p \mid  \bot \mid \top \mid  \phi\land \phi \mid \phi\dor\phi\mid  {=_n\hspace{-0.1cm}(p_1,\ldots,p_{n})}\mid \filleddiamond \phi\tag{$\CONDEP$}\\
    &\phi\mathrel{::=} p  \mid \bot \mid  \phi\land \phi \mid \phi\to\phi \mid \phi \vvee\phi \tag{$\INQ$}\\
    &\phi\mathrel{::=} p  \mid \bot \mid  \phi\land \phi \mid \phi\to\phi \mid \phi \vveedot\phi\mid \filleddiamond \phi \tag{$\CONINQ$}\\
    &\phi\mathrel{::=} p  \mid \bot \mid  \phi\land \phi \mid \phi\to\phi \mid \filleddiamond \phi \tag{$\PLIM$}
    \end{align*}
    where $n\in \mathbb{N}$ and $p_1,\ldots,p_n\in \mathsf{P}$.
\end{definition}

 We use the first Greek letters $\alpha$ and $\beta$ to range exclusively over classical formulas (formulas of $\PLV$/$\PLD$/$\PLI$). We write $\mathsf{P}(\phi)$ for the set of propositional letters appearing in $\phi$, and $\phi(\mathsf{X})$ if $\mathsf{P}(\phi)\subseteq \mathsf{X}\subseteq\mathsf{P}$. We write $\phi (\psi_1/p_1,\ldots,\psi_n/p_n)$ for the result of replacing all occurrences of $p_i$ in $\phi$ by $\psi_i$, for $1\leq i \leq n$. We also use $\PLV$ to refer to the set of formulas of $\PLV$, and similarly for the other logics.

 A (propositional) \emph{team $t$ with domain $\mathsf{X}\subseteq\mathsf{P}$} (or a \emph{team over $\mathsf{X}$}) is a set of valuations with domain $\mathsf{X}$: $t\subseteq 2^\mathsf{X}$.

\begin{definition}[Semantics]
\label{c:definition:semantics_PLNE}
    Given a team $t$ over $\mathsf{X}$, the \emph{truth} of a formula $\phi(\mathsf{X})$ in $t$ (written $t\vDash \phi$) is defined by the following recursive clauses:
    \begin{align*}
        &t\vDash p && :\iff && v(p)=1 \textit{ for all $v\in t$};\\
        &t\vDash \lnot p && :\iff && v(p)=0 \textit{ for all $v\in t$};\\
        &t\vDash \bot && :\iff && t=\varnothing;\\
        &t\vDash \top ;&&  && \\
        &t\vDash \varphi\land\psi && :\iff  && t\vDash \varphi \textit{\hspace{0.1cm} and \hspace{0.1cm}} t\vDash \psi;\\
        &t\vDash \varphi\lor\psi && :\iff  && \textit{there exist $s,u$ such that } s\vDash \varphi,\textit{ } u\vDash \psi,\textit{ and $t= s\cup u$;}\\
        &t\vDash \varphi\dor\psi && :\iff  && \textit{there exist $s,u$ such that } s\vDash \varphi,\textit{ } u\vDash \psi,\textit{ and $t\subseteq s\cup u$;}\\
        &t\vDash \phi\to\psi && :\iff  && \textit{for all }s\subseteq t:\textit{ if } s\vDash \varphi \textit{ then }s\vDash \psi;\\
         &t\vDash {=_n\hspace{-0.1cm}(p_1,\ldots,p_{n})} && :\iff &&\textit{for all $v,w\in t{:}$ if }[\forall 1\leq i\leq n-1:v(p_i)=w(p_i) ]\\
         &&&&&\textit{ then } v(p_n)=w(p_n);\\
        &t\vDash \filleddiamond\phi&& :\iff && \textit{there exists $s\subseteq t$ such that $s\neq \emptyset$ and $s\vDash \phi$};\\
        &t\vDash \phi\vvee\psi && :\iff  && t\vDash \phi\textit{ or } t\vDash \psi;\\
        &t\vDash \phi\vveedot\psi && :\iff  &&\textit{there exists }s\supseteq t\textit{ such that } s\vDash \phi\text{ or } s\vDash \psi.
    \end{align*}
    We say that a set of formulas $\Gamma$ \emph{entails} $\phi$, written $\Gamma\models \varphi$, if for all teams $t$, if $t\models \gamma$ for all $\gamma\in \Gamma$, then $t\models \varphi$. We write simply $\phi_1,\ldots,\phi_n\models\phi$ for $\{\phi_1,\ldots,\phi_n\}\models\phi$ and $\models \phi$ for $\emptyset\models \phi$, where $\emptyset$ refers to the empty set of formulas (i.e. not to the empty team). If both $\phi\models\psi$ and $\psi\models\phi$, we say that $\phi$ and $\psi$ are \emph{equivalent}, and write $\phi \equiv \psi$.\footnote{It should be noted that some literature on propositional inquisitive logic (e.g., \citep{ciardellibook}) employs a semantics different from that presented here: formulas are instead interpreted with respect to sets of possible worlds in \emph{propositional information models} $(W,V)$, where $W$ is a set of worlds and $V:W\times \mathsf{X}\to \{0,1\}$ (where $\mathsf{X}\subseteq \mathsf{P}$) is a valuation function (cf. Section \ref{c:section:convex_modal}). It makes no difference which of these semantics is chosen for $\INQ$. However, the choice between these two semantics is significant for the logics $\CONINQ$ and $\CONDEP$. It is easy to see, for instance, that $ \models  \filleddiamond p \vveedot  \filleddiamond p$ within our semantics, whereas this formula is not valid over propositional information models. See also Fact \ref{c:fact:failure_of_bisim_invariance}. We are grateful to Fausto Barbero for pointing this out to us.}
\end{definition}
Further, for later convenience, we define:
\begin{itemize}
    \item In $\PLI$ and its extensions: $\lnot \phi:=\phi \to \bot$ and $\top:=\lnot \bot$; 
    \item In $\PLV$ and $\PLD$: $\lnot \alpha$ is an abbreviation for the negation normal form of $\alpha$ (e.g., $\lnot ( p \land \lnot q)$ in $\PLD$ abbreviates $\lnot p \veedot q$);
    \item In languages with $\filleddiamond$: $\Bot\,\mathrel{:=}\filleddiamond \bot$;
    \item In the classical bases: $\alpha\curlyvee \beta:=\lnot (\lnot \alpha \land \lnot \beta)$ (so in $\PLV$, $\alpha\curlyvee \beta=\alpha \vee\beta$; and in $\PLD$, $\alpha\curlyvee \beta=\alpha \veedot\beta$).
\end{itemize}  
Note that in $\PLI$ and its extensions, $t\vDash \lnot \phi\iff \forall s\subseteq t: [s\vDash\phi$ implies $s=\emptyset]$. Note also that $t\vDash \top$ is always the case and $t\vDash \Bot$ is never the case, and that $t\vDash\alpha\curlyvee \beta$ iff for all $v\in t$, $\{v\}\vDash \alpha$ or $\{v\}\vDash \beta$. We stipulate that $\bigvee \emptyset:=\bot$, $\bigveedot\emptyset:=\bot$, $\bigwedge \emptyset:=\top$, $\bigvvee \emptyset:=\bot$, $\bigvveedot\emptyset:=\bot$.
\\\\
With these definitions in place, a few remarks on the connectives and their clauses are in order.

  \begin{figure}[t]  
\centering
\begin{subfigure}[b]{0.24\textwidth}
\centering
    \begin{tikzpicture}[>=latex,scale=.7]
    
\draw[opaque, rounded corners] (-1.8,1.8) rectangle (1.8, 0.2);
    
    \draw (1,1) node[index gray, minimum size=0.8cm] (wp) {$v_{p\overline{q}}$};
    \draw (-1,1) node[index gray, minimum size=0.8cm] (wpq) {$v_{pq}$};
    \draw (-1,-1) node[index gray, minimum size=0.8cm] (wq) {$v_{\overline{p}q}$};
    \draw (1,-1) node[index gray, minimum size=0.8cm] (w4) {$v_{\overline{pq}}$};
  	
  	\end{tikzpicture}
    \caption{$\{v_{pq},v_{p\overline{q}}\}\vDash p\vee q $; $\{v_{pq},v_{p\overline{q}}\}\vDash \filleddiamond q$}
    \label{c:figure:vee_might}
    \end{subfigure}
\hspace{0.5cm}
\begin{subfigure}[b]{0.3\textwidth}
\centering
\begin{tabular}{|c|c|c|c|}
\hline 
& p & q & r \\ 
\hline 
$v_1$&\cellcolor{gray!50}$0$&\cellcolor{gray!50}$1$&\cellcolor{red!50}$1$\\
\hline 
$v_2$&\cellcolor{gray!50}$0$&\cellcolor{gray!50}$1$&\cellcolor{red!50}$0$\\
\hline 
$v_3$&\cellcolor{gray!75}$1$&\cellcolor{gray!75}0&$0$\\
\hline 
$v_4$&\cellcolor{gray!75}$1$&\cellcolor{gray!75}$0$&$0$\\
\hline 
\end{tabular}
\caption{$\{v_1,v_2,v_3,v_4\}\vDash \dep{p}{q}$; $\{v_1,v_2,v_3,v_4\}\nvDash \dep{p}{r}$}
\label{c:figure:dep}
\end{subfigure}
\hspace{0.5cm}
\begin{subfigure}[b]{0.24\textwidth}
\centering
    \begin{tikzpicture}[>=latex,scale=.7]
    
\draw[opaque, rounded corners, color=blue] (-1.8,1.8) rectangle (1.8, 0.2);
   \draw[opaque, rounded corners, color=red] (-1.8,-0.2) rectangle (1.8, -1.8);
    
    \draw (1,1) node[index gray, minimum size=0.8cm] (wp) {$v_{p\overline{q}}$};
    \draw (-1,1) node[index gray, minimum size=0.8cm] (wpq) {$v_{pq}$};
    \draw (-1,-1) node[index gray, minimum size=0.8cm] (wq) {$v_{\overline{p}q}$};
    \draw (1,-1) node[index gray, minimum size=0.8cm] (w4) {$v_{\overline{pq}}$};
  	
  	\end{tikzpicture}
    \caption{$\{v_{pq},v_{p\overline{q}}\}\vDash p\vvee q$;
    $\{v_{\overline{p}q},v_{\overline{pq}}\}\nvDash p\vvee q$}
        \label{c:figure:inq}
    \end{subfigure}
\caption{Examples of the semantics.}
    	\end{figure}

A split disjunction $\phi \vee\psi$ is true in a team $t$ just in case the team can be split into subteams $s$ and $u$ such that $\phi$ is true in $s$ and $\psi$ is true in $u$. See Figure \ref{c:figure:vee_might}. Note that one of the two subteams can be empty; for instance, we have that $\{v_{p\overline{q}}\}\vDash p\vee \lnot p$ (where $v_{p\overline{q}}(p)=1$ and $v_{p\overline{q}}(q)=0$) because $\{v_{p\overline{q}}\}=\{v_{p\overline{q}}\}\cup \emptyset$ where $\{v_{p\overline{q}}\}\vDash p$ and $\emptyset\vDash \lnot p$. $\PLV$ is the standard classical propositional basis in the dependence logic literature (being used, for instance, in dependence logic \cite{vaananen2007} and inclusion logic \cite{galliani2012}), and $\vee$ is the canonical ``classical'' (in the sense of Fact \ref{c:fact:classical_correspondence}, below) disjunction in this literature. Propositional dependence logic $\DEP$ extends $\PLV$ with the functional dependence connectives $=_n\hspace{-0.1cm}(\cdot)$, yielding dependence atoms $=_n\hspace{-0.1cm}(p_1,\ldots,p_n)$---we usually write simply $\dep{p_1,\ldots}{p_n}$. The intuitive meaning of $\dep{p_1,\ldots}{p_n}$ is that the truth values of $p_1,\ldots,p_{n-1}$ jointly determine the truth value of $p_n$. See Figure \ref{c:figure:dep}. A unary dependence atom $\con{p}$ is called a \emph{constancy atom}---$\con{p}$ is true in a team just in case the truth value of $p$ is constant across the team ($t\vDash {\con{p}}$ iff $[\forall v\in t: v(p)=1$ or $\forall v\in t:v(p)=0]$). The nullary dependence atom $=\hspace{-0.1cm}()$ is equivalent to $\top$.

The classical basis $\PLD$ replaces the split disjunction $\vee$ of $\PLV$ with the variant $\dor$, where $\phi \veedot \psi$ is true in a team $t$ just in case $t$ is a subteam of the union of some $s$ which makes $\phi$ true and some $u$ which makes $\psi$ true.\footnote{Just as the semantic clause for the split disjunction $\vee$ is not unique to team semantics but has been employed across many logical systems (including fusion, product or multiplicative conjunction in substructural logics \cite{andersonbelnapdunn92, urquhart72}, conjunction in truthmaker semantics \cite{fine17, fraassen69}, and the modality in modal information logic \cite{knudstorp22, knudstorp23, vanbenthem96}), the clause for the ideal disjunction $\veedot$ has found application within multiple logical systems, particularly in semantics for non-distributive logics \cite{bezhetal24,dosen89}. And in fact, while the split disjunction has its origin in Hodges' trump semantics for independence-friendly logic \citep{hodges1997}, the semantic clause employed by Hodges was essentially that for the ideal disjunction (the two disjunctions are equivalent in a downward-closed setting such as that of independence-friendly logic---see Section \ref{c:section:closure_props}).} We call $\dor$ the \textit{ideal disjunction}, as its clause defines the \textit{ideal} generated by two ideals. To explicate, recall that given a join-semilattice $(S, \curlyvee)$ and ideals $A,B\subseteq S$, the least ideal containing $A$ and $B$ is $\{s\in S\mid \exists a\in A, b\in B:s\leq a\curlyvee b\}$.  In particular, for a powerset $\wp(X)$ ordered by inclusion `$\subseteq$' and ideals $A,B\subseteq \wp(X)$, the ideal generated by $A$ and $B$ is $\{s\in \wp(X)\mid \exists a\in A, b\in B:s\subseteq a\cup b\}$. So, writing $\left \Vert\varphi\right \Vert$ for the teams in which $\varphi$ is true, given $\left \Vert\varphi\right\Vert$ and $\left \Vert\psi\right\Vert$ ideals, $\left \Vert\varphi\veedot \psi\right\Vert=\{s\mid \exists a\in \left \Vert\varphi\right\Vert, b\in \left \Vert\psi\right\Vert:s\subseteq a\cup b\}$ is the ideal generated by $\left \Vert\varphi\right\Vert$ and $\left \Vert\psi\right\Vert$. (Note that for distributive join-semilattices, the operators $\vee$ and $\veedot$ coincide on down-sets, so also on ideals.) Important for our purposes is that $\veedot$ preserves convexity, whereas $\vee$ does not, as we will see below.

Our convex variant of dependence logic $\CONDEP$ extends $\PLD$ with the epistemic might-operator $\filleddiamond$, where $\filleddiamond \phi$ is true in $t$ just in case $t$ contains a nonempty subteam in which $\phi$ is true. See Figure \ref{c:figure:vee_might}.

$\PLI$, featuring the \emph{intuitionistic implication}---where $t\vDash \phi\to\psi$ holds just in case whenever $\phi$ is true in a subteam of $t$, so is $\psi$---is the classical basis for propositional inquisitive logic $\INQ$. $\INQ$ extends $\PLI$ with the inquisitive disjunction $\vvee$ which has the standard Boolean disjunction satisfaction clause (with respect to teams), but which, as we will see, behaves nonclassically. The inquisitive disjunction is used to model the meanings of questions in inquisitive logic and semantics. For instance, $p\vvee q$ represents the question `$p$ or $q$?'---the question is \emph{resolved} in a team just in case one of its answers is true. See Figure \ref{c:figure:inq}. $\CONINQ$ replaces the inquisitive disjunction $\vvee$ of $\INQ$ with $\vveedot$, which, similarly to $\veedot$, guarantees the preservation of convexity, whereas $\vvee$ does not. Like the ideal disjunction, $\vveedot$ corresponds to a natural mathematical operation: $\left\Vert\varphi\vveedot \psi\right\Vert=\{s\mid \exists a\in \left\Vert\varphi\right\Vert\cup \left\Vert\psi\right\Vert: s\subseteq a \}$ is the \emph{down-set} generated by $\left\Vert\varphi\right\Vert$ and $\left\Vert\psi\right\Vert$. Thus we call it the \textit{downjunction}. $\PLIM$ is the $\vveedot$-free fragment of $\CONINQ$; we will show that it is already sufficiently expressive to capture all convex properties.

\subsubsection{Preliminaries: Closure Properties} \label{c:section:closure_props}
We have repeatedly stated that (a) the fragments $\PLV, \PLD$, and $\PLI$ are ``classical''; and (b) our logics $\CONDEP$, $\CONINQ,$ and $\PLIM$ are ``convex''. To justify this terminology, we need the concept of a team-semantical closure property:

\begin{definition}[Closure properties]\label{c:definition:closure_props}
    Given any formula $\varphi$, we say that
    \begin{align*}
        &\varphi \text{ is  \textit{downward closed}}&& \text{iff} &&\left[t\vDash \varphi \text{ and } s \subseteq t\right] \Rightarrow s \vDash \varphi;\\
        &\varphi \text{ is  \textit{upward closed}}&& \text{iff} &&\left[t\vDash \varphi \text{ and } s \supseteq t\right] \Rightarrow s \vDash \varphi;\\
        &\varphi \text{ is  \textit{convex}}&& \text{iff} &&\left[t\vDash \varphi ,\text{ } s \vDash \varphi\text{, and }s\subseteq u\subseteq t\right] \Rightarrow u \vDash \varphi;\\
        &\varphi \text{ is \textit{union closed}}&& \text{iff} &&\left[t\vDash \varphi \text{ for all } t\in T\neq \emptyset\right] \Rightarrow \bigcup T  \vDash \varphi;\\
        &\varphi \text{ has the \textit{empty team property}}&& \text{iff}&& \varnothing \vDash \varphi ;\\
        &\varphi \text{ is \textit{flat}}&& \text{iff}&& t\vDash\varphi  \Leftrightarrow \left[\{v\}\vDash \varphi\text{ for all }v\in t\right].
    \end{align*}
\end{definition}

One can then readily verify the following facts:

\begin{fact}\label{c:fact:closure_props}\ For all formulas $\phi(\mathsf{X})$ over the propositional letters $\mathsf{X}$:
\begin{enumerate}
    \item[(i)] $\phi$ is flat iff $\phi$ is downward and union closed, and $\phi$ has the empty team property.
    \item[(ii)] $\phi$ is downward closed iff $\phi$ is convex and $[$if $\phi\nequiv \Bot$, then $\emptyset\vDash \phi ]$.
    \item[(iii)] $\phi$ is upward closed iff $\phi$ is convex and $[$if $\phi\nequiv \Bot$, then $2^{\mathsf{X}}\vDash \phi ]$.
\end{enumerate}
\end{fact}

\begin{fact} \label{c:fact:classical_correspondence} All classical formulas (formulas of $\PL_\vee$/$\PLD$/$\PLI$) are flat and their team semantics are equivalent to their usual single-valuation semantics on singletons; that is, for all classical $\alpha$:
    $$t\vDash \alpha \iff \{v\}\vDash \alpha \text{ for all }v \in t \iff v(\alpha)=1 \text{ for all }v \in t.$$
\end{fact}
This explains (a) why we refer to $\PLV, \PLD$, and $\PLI$ as ``classical bases'': from Fact \ref{c:fact:classical_correspondence}, it follows that for any set $\Gamma\cup \{\alpha\}$ of classical formulas,
    $$\Gamma\vDash \alpha \iff \Gamma\vDash_{\mathrm{CPL}} \alpha,$$
where $\vDash_{\mathrm{CPL}} $ is the usual entailment relation for single-valuation semantics. Given these facts, we use the notations $\{v\}\vDash \alpha$ and $v\vDash \alpha$ (and $v(\alpha)=1$) interchangeably whenever $\alpha$ is classical, and similarly for $\Gamma\vDash \alpha$ and $\Gamma\vDash_{\mathrm{CPL}} \alpha$.

As for (b) why we call the logics convex, we have the following proposition:

    \begin{proposition} \label{c:prop:convex_logics_convex}
        Formulas of $\DEP$ and $\INQ$ are downward closed (and hence also convex). Formulas of $\CONDEP$, $\CONINQ$, and $\PLIM$ are convex.
    \end{proposition}
    \begin{proof}
        Proved by induction on the structure of formulas $\phi$. Most cases are straightforward---note in particular that $\dep{p_1,\ldots}{p_n}$, $\phi\to\psi$, $\phi\veedot \psi$, and $\phi\vveedot \psi$ are always downward closed, and hence convex, and that $\filleddiamond \phi$ is always upward closed, and hence convex.
    \end{proof}
    Later, we show that the converse holds as well: every convex class of teams is expressible by a formula of $\CONDEP$/$\CONINQ$/$\PLIM$ (that every (nonempty) downward closed class of teams is expressible by a formula of $\DEP$/$\INQ$, is, as mentioned, already known). 
    
    But before that, we comment on the other closure properties, and on our choice of $\veedot$ and $\vveedot$ over the usual $\vee$ and $\vvee$ as regards expressivity.

It is easy to see that formulas with the might-operator $\filleddiamond$ may violate downward closure and the empty team property. Likewise, union closure fails in these logics, for consider the formula $\filleddiamond p\to q$. We have that $\{v_{pq}\}\vDash \filleddiamond p\to q$ and $\{v_{\overline{p}\overline{q}}\}\vDash \filleddiamond p\to q$, but $\{v_{pq},v_{\overline{p}\overline{q}}\}\nvDash \filleddiamond p\to q$. Similarly, for all $\phi\in\{\dep{q}{p},p \vvee \lnot p,p \vveedot \lnot p\}$, we have $\{v_{pq}\}\vDash \phi$ and $\{v_{\overline{p}q}\}\vDash \phi$, but $\{v_{pq},v_{\overline{p}q}\}\nvDash \phi$.

As mentioned above, whereas the ideal disjunction  $\veedot$ always preserves convexity, the split disjunction $\vee$ need not. Ditto for $\vvee$ and $\vveedot$.

\begin{fact} \label{c:fact:disjunctions_convexity_break}
    There are convex $\phi,\psi$ such that $\phi\vee\psi$ is not convex. Similarly, there are convex $\phi',\psi'$ such that $\phi'\vvee\psi'$ is not convex.
\end{fact}
\begin{proof}
We work in subteams of $\{v_{pq},v_{p\overline{q}},v_{\overline{p}q}\}$---see Figure \ref{c:figure:failure_of_convexity} for a depiction of the denotations $\left\Vert\cdot\right\Vert$ of the formulas in the powerset lattice of $\{v_{pq},v_{p\overline{q}},v_{\overline{p}q}\}$. (See also the similar Fact \ref{c:fact:disjunctions_convexity_break2} and its simpler but more abstract proof.)

  \begin{figure}[t]  
\centering
\begin{subfigure}[t]{0.49\textwidth}
\centering
\begin{tikzpicture}[scale=0.58]
\tikzset{every node/.style={fill=none,draw=black,circle,inner sep=2pt}}
\tikzset{every edge/.style={draw,-latex}}
\tikzset{every loop/.style={min distance=10mm,in=0,out=60,looseness=10}}

\node[fill,draw=black,circle,inner sep=1pt] (N) at (0,-4) {};
\node[fill,draw=black,circle,inner sep=1pt] (xyz) at (0,2) {};

\node[fill,draw=black,circle,inner sep=1pt] (x) at (-2,-2) {};
\node[fill,draw=black,circle,inner sep=1pt] (y) at (0,-2) {};
\node[fill,draw=black,circle,inner sep=1pt] (z) at (2,-2) {};
\node[draw=none,inner sep=0pt] (a) at (0.4,-4.4) {};

\node[fill,draw=black,circle,inner sep=1pt] (xy) at (-2,0) {};
\node[fill,draw=black,circle,inner sep=1pt] (xz) at (0,0) {};
\node[fill,draw=black,circle,inner sep=1pt] (yz) at (2,0) {};

\draw (N) edge (x);
\draw (N) edge (y);
\draw (N) edge (z);

\draw (x) edge (xy);
\draw (x) edge (xz);
\draw (y) edge (xy);
\draw (y) edge (yz);
\draw (z) edge (xz);
\draw (z) edge (yz);

\draw (xy) edge (xyz);
\draw (xz) edge (xyz);
\draw (yz) edge (xyz);

  \draw[fill=blue, fill opacity=0.1,rounded corners] plot[] coordinates {([shift={(-13pt,-5pt)}]x) ([shift={(-13pt,13pt)}]x) ([shift={(-13pt,13pt)}]xz) ([shift={(13pt,13pt)}]xz) ([shift={(13pt,-13pt)}]xz) ([shift={(13pt,13pt)}]y) ([shift={(13pt,-13pt)}]y) ([shift={(13pt,-13pt)}]x) ([shift={(-13pt,-13pt)}]x) ([shift={(-13pt,-5pt)}]x)};

\end{tikzpicture}
\caption{$\left\Vert\filleddiamond p \land \con{q}\right\Vert$ (convex)}
\end{subfigure}
\begin{subfigure}[t]{0.49\textwidth}
\centering
\begin{tikzpicture}[scale=0.58]
\tikzset{every node/.style={fill=none,draw=black,circle,inner sep=2pt}}
\tikzset{every edge/.style={draw,-latex}}
\tikzset{every loop/.style={min distance=10mm,in=0,out=60,looseness=10}}

\node[fill,draw=black,circle,inner sep=1pt] (N) at (0,-4) {};
\node[fill,draw=black,circle,inner sep=1pt] (xyz) at (0,2) {};

\node[fill,draw=black,circle,inner sep=1pt] (x) at (-2,-2) {};
\node[fill,draw=black,circle,inner sep=1pt] (y) at (0,-2) {};
\node[fill,draw=black,circle,inner sep=1pt] (z) at (2,-2) {};
\node[draw=none,inner sep=0pt] (a) at (0.4,-4.4) {};

\node[fill,draw=black,circle,inner sep=1pt] (xy) at (-2,0) {};
\node[fill,draw=black,circle,inner sep=1pt] (xz) at (0,0) {};
\node[fill,draw=black,circle,inner sep=1pt] (yz) at (2,0) {};

\draw (N) edge (x);
\draw (N) edge (y);
\draw (N) edge (z);

\draw (x) edge (xy);
\draw (x) edge (xz);
\draw (y) edge (xy);
\draw (y) edge (yz);
\draw (z) edge (xz);
\draw (z) edge (yz);

\draw (xy) edge (xyz);
\draw (xz) edge (xyz);
\draw (yz) edge (xyz);

  \draw[fill=red, fill opacity=0.1,rounded corners] plot[] coordinates {([shift={(-13pt,-5pt)}]x) ([shift={(-13pt,13pt)}]x) ([shift={(-13pt,13pt)}]xy) ([shift={(-13pt,13pt)}]xyz) ([shift={(13pt,13pt)}]xyz) ([shift={(13pt,-13pt)}]xyz) ([shift={(13pt,13pt)}]xz) ([shift={(13pt,-13pt)}]xz) ([shift={(13pt,13pt)}]y) ([shift={(13pt,-13pt)}]y)  ([shift={(13pt,-13pt)}]x) ([shift={(-13pt,-13pt)}]x) ([shift={(-13pt,-5pt)}]x)};

\end{tikzpicture}
\caption{$\left\Vert(\filleddiamond p \land\con{q} )\vee(\filleddiamond p \land\con{q})\right\Vert$ (not convex)}
\end{subfigure}\\
\begin{subfigure}[t]{0.3\textwidth}
\centering
\begin{tikzpicture}[scale=0.58]
\tikzset{every node/.style={fill=none,draw=black,circle,inner sep=2pt}}
\tikzset{every edge/.style={draw,-latex}}
\tikzset{every loop/.style={min distance=10mm,in=0,out=60,looseness=10}}

\node[fill,draw=black,circle,inner sep=1pt] (N) at (0,-4) {};
\node[fill,draw=black,circle,inner sep=1pt] (xyz) at (0,2) {};

\node[fill,draw=black,circle,inner sep=1pt] (x) at (-2,-2) {};
\node[fill,draw=black,circle,inner sep=1pt] (y) at (0,-2) {};
\node[fill,draw=black,circle,inner sep=1pt] (z) at (2,-2) {};
\node[draw=none,inner sep=0pt] (a) at (0.4,-4.4) {};

\node[fill,draw=black,circle,inner sep=1pt] (xy) at (-2,0) {};
\node[fill,draw=black,circle,inner sep=1pt] (xz) at (0,0) {};
\node[fill,draw=black,circle,inner sep=1pt] (yz) at (2,0) {};

\draw (N) edge (x);
\draw (N) edge (y);
\draw (N) edge (z);

\draw (x) edge (xy);
\draw (x) edge (xz);
\draw (y) edge (xy);
\draw (y) edge (yz);
\draw (z) edge (xz);
\draw (z) edge (yz);

\draw (xy) edge (xyz);
\draw (xz) edge (xyz);
\draw (yz) edge (xyz);

   \draw[fill=blue, fill opacity=0.1,rounded corners] plot[] coordinates {([shift={(-13pt,-5pt)}]x) ([shift={(-13pt,3pt)}]x) ([shift={(-13pt,13pt)}]xy) ([shift={(13pt,13pt)}]xy) ([shift={(13pt,13pt)}]y) ([shift={(13pt,-13pt)}]N) ([shift={(-13pt,-13pt)}]N) ([shift={(-13pt,-13pt)}]x) ([shift={(-13pt,-5pt)}]x)};

\end{tikzpicture}
\caption{$\left\Vert p\right\Vert$ (convex)}
\end{subfigure}
\begin{subfigure}[t]{0.3\textwidth}
\centering
\begin{tikzpicture}[scale=0.58]
\tikzset{every node/.style={fill=none,draw=black,circle,inner sep=2pt}}
\tikzset{every edge/.style={draw,-latex}}
\tikzset{every loop/.style={min distance=10mm,in=0,out=60,looseness=10}}

\node[fill,draw=black,circle,inner sep=1pt] (N) at (0,-4) {};
\node[fill,draw=black,circle,inner sep=1pt] (xyz) at (0,2) {};

\node[fill,draw=black,circle,inner sep=1pt] (x) at (-2,-2) {};
\node[fill,draw=black,circle,inner sep=1pt] (y) at (0,-2) {};
\node[fill,draw=black,circle,inner sep=1pt] (z) at (2,-2) {};
\node[draw=none,inner sep=0pt] (a) at (0.4,-4.4) {};

\node[fill,draw=black,circle,inner sep=1pt] (xy) at (-2,0) {};
\node[fill,draw=black,circle,inner sep=1pt] (xz) at (0,0) {};
\node[fill,draw=black,circle,inner sep=1pt] (yz) at (2,0) {};

\draw (N) edge (x);
\draw (N) edge (y);
\draw (N) edge (z);

\draw (x) edge (xy);
\draw (x) edge (xz);
\draw (y) edge (xy);
\draw (y) edge (yz);
\draw (z) edge (xz);
\draw (z) edge (yz);

\draw (xy) edge (xyz);
\draw (xz) edge (xyz);
\draw (yz) edge (xyz);

\filldraw[
    fill=blue,
    fill opacity=0.1,
    draw=black,
    even odd rule,
    rounded corners
] 
plot[] coordinates {
    ([shift={(-5pt,-8pt)}]y)
    ([shift={(-13pt,0pt)}]xy)
    ([shift={(0pt,13pt)}]xyz)
    ([shift={(13pt,0pt)}]yz)
    ([shift={(0pt,-13pt)}]y)
    ([shift={(-5pt,-8pt)}]y)
}
plot[] coordinates {
    ([shift={(-10pt,23pt)}]y)
    ([shift={(13pt,0pt)}]xy)
    ([shift={(0pt,-13pt)}]xyz)
    ([shift={(-13pt,0pt)}]yz)
    ([shift={(0pt,13pt)}]y)
    ([shift={(-10pt,23pt)}]y)
};

\end{tikzpicture}
\caption{$\left\Vert\filleddiamond \lnot q\right\Vert$ (convex)}
\end{subfigure}
\begin{subfigure}[t]{0.3\textwidth}
\centering
\begin{tikzpicture}[scale=0.58]
\tikzset{every node/.style={fill=none,draw=black,circle,inner sep=2pt}}
\tikzset{every edge/.style={draw,-latex}}
\tikzset{every loop/.style={min distance=10mm,in=0,out=60,looseness=10}}

\node[fill,draw=black,circle,inner sep=1pt] (N) at (0,-4) {};
\node[fill,draw=black,circle,inner sep=1pt] (xyz) at (0,2) {};

\node[fill,draw=black,circle,inner sep=1pt] (x) at (-2,-2) {};
\node[fill,draw=black,circle,inner sep=1pt] (y) at (0,-2) {};
\node[fill,draw=black,circle,inner sep=1pt] (z) at (2,-2) {};
\node[draw=none,inner sep=0pt] (a) at (0.4,-4.4) {};

\node[fill,draw=black,circle,inner sep=1pt] (xy) at (-2,0) {};
\node[fill,draw=black,circle,inner sep=1pt] (xz) at (0,0) {};
\node[fill,draw=black,circle,inner sep=1pt] (yz) at (2,0) {};

\draw (N) edge (x);
\draw (N) edge (y);
\draw (N) edge (z);

\draw (x) edge (xy);
\draw (x) edge (xz);
\draw (y) edge (xy);
\draw (y) edge (yz);
\draw (z) edge (xz);
\draw (z) edge (yz);

\draw (xy) edge (xyz);
\draw (xz) edge (xyz);
\draw (yz) edge (xyz);

 \filldraw[
    fill=red,
    fill opacity=0.1,
    draw=black,
    even odd rule,
    rounded corners
] 
plot[] coordinates {
    ([shift={(-13pt,-5pt)}]xy)
    ([shift={(-13pt,0pt)}]xy)
    ([shift={(0pt,13pt)}]xyz)
    ([shift={(13pt,0pt)}]yz)
    ([shift={(13pt,0pt)}]y)
    ([shift={(13pt,-13pt)}]N) ([shift={(-13pt,-13pt)}]N)
    ([shift={(-13pt,-13pt)}]x)
    ([shift={(-13pt,-5pt)}]xy)
}
plot[] coordinates {
    ([shift={(-5pt,18pt)}]y)
    ([shift={(13pt,0pt)}]xy)
    ([shift={(0pt,-13pt)}]xyz)
    ([shift={(-13pt,0pt)}]yz)
    ([shift={(0pt,13pt)}]y)
    ([shift={(-5pt,18pt)}]y)
};

\end{tikzpicture}
\caption{$\left\Vert p\vvee \filleddiamond \lnot q\right\Vert$ (not convex)}
\end{subfigure}
\caption{Pictorial presentation of Fact~\ref{c:fact:disjunctions_convexity_break}: denotations depicted in the powerset lattice of $\{v_{pq},v_{p\overline{q}},v_{\overline{p}q}\}$.}
\label{c:figure:failure_of_convexity}
\end{figure}

For the first part, let $\varphi\mathrel{:=}\filleddiamond p\land\con{q}$. Clearly the first conjunct is upward closed (and hence convex) and the second is downward closed (and hence convex); therefore, since conjunction preserves convexity, $\varphi$ is convex. We show that $\{v_{p\overline{q}}\}\vDash \varphi\vee\varphi$; $\{v_{p\overline{q}},v_{\overline{p}q}\}\nvDash \varphi\vee\varphi$; and $\{v_{pq},v_{p\overline{q}},v_{\overline{p}q}\}\vDash \varphi\vee\varphi$, which will suffice as $\{v_{p\overline{q}}\}\subseteq \{v_{p\overline{q}},v_{\overline{p}q}\}\subseteq \{v_{pq},v_{p\overline{q}},v_{\overline{p}q}\}$. 

Since $\{v_{p\overline{q}}\}\vDash p$ and  $\{v_{pq}\}\vDash p$, we have $\{v_{p\overline{q}}\}\vDash \filleddiamond p$ and $\{v_{pq},v_{\overline{p}q}\}\vDash \filleddiamond p$. And since the value of $q$ is constant in both $\{v_{p\overline{q}}\}$ and $\{v_{pq},v_{\overline{p}q}\}$, they both make $\con{q}$ true. Thus $\{v_{p\overline{q}}\}\vDash \varphi$ and $\{v_{pq},v_{\overline{p}q}\}\vDash \varphi$, whence $\{v_{p\overline{q}}\}=\{v_{p\overline{q}}\}\cup \{v_{p\overline{q}}\}\vDash \varphi\lor\varphi$ and $\{v_{pq},v_{p\overline{q}},v_{\overline{p}q}\}=\{v_{p\overline{q}}\}\cup \{v_{pq},v_{\overline{p}q}\}\vDash \varphi\lor\varphi$. On the other hand, no subteam of $\{v_{p\overline{q}},v_{\overline{p}q}\}$ containing $v_{\overline{p}q}$ makes $\varphi$ true: because if $v_{\overline{p}q}\in t\subseteq \{v_{p\overline{q}},v_{\overline{p}q}\}$, then $t=\{v_{p\overline{q}},v_{\overline{p}q}\}\nvDash\, \con{q}$ or $t=\{v_{\overline{p}q}\}\nvDash \filleddiamond p$. Thus $\{v_{p\overline{q}},v_{\overline{p}q}\}\nvDash \varphi\lor \varphi$.

\sloppy
For the second part, note that $p$ and $\filleddiamond \lnot q$ are convex. We have $\{v_{pq}\}\vDash p\vvee \filleddiamond \lnot q$; $\{v_{pq},v_{\overline{p}q}\}\nvDash p\vvee \filleddiamond  \lnot q $; and $\{v_{pq},v_{p\overline{q}},v_{\overline{p}q}\}\vDash p\vvee \filleddiamond \lnot q$, where $\{v_{pq}\}\subseteq \{v_{pq},v_{\overline{p}q}\}\subseteq \{v_{pq},v_{p\overline{q}},v_{\overline{p}q}\}$.
\end{proof}
\fussy

 It follows that no logic expressively complete for the class of all convex properties can incorporate $\vee$---this is why we swap $\vee$ for $\veedot $ in our convex variant of dependence logic $\CONDEP$. The situation with $\vvee$ and $\vveedot$ is analogous.\footnote{\label{endnote:convex_closure_disjunctions}While the choice of $\veedot$ and $\vveedot$ is well-motivated, there are, of course, other options as well; for example, instead of $\veedot$ one could also opt for a connective whose semantic clause is, instead of the downward closure, the convex closure of the split disjunction (this is similar to a clause considered by Marius Tritschler in unpublished work~\cite{tritschler2025}):
 $$t\vDash \phi \vee_C\psi\iff \text{there are } s_1,s_2,u_1,u_2 \text{ s.t. } s_i\vDash \phi,u_i\vDash \psi \text{ and } s_1\cup u_1 \subseteq t\subseteq s_2\cup u_2 .$$
 Likewise for $\vvee$:
 $$t\vDash \phi \vvee_C\psi\iff \text{there are } s_1,s_2, \text{ s.t. } s_i\vDash \phi \text{ or }s_i\vDash \psi \text{ for $i\in \{1,2\}$, and } s_1 \subseteq t\subseteq s_2.$$}

 We end this section with an aside, recording that formulas of the convex logics have the \emph{locality property}, defined as follows.

\begin{definition}[Locality]\label{def:locality}
    A formula $\phi(\mathsf{X})$ has the \emph{locality property} if for any team $t$ over $\mathsf{Y}\supseteq \mathsf{X}$ and $s$ over $\mathsf{Z}\supseteq \mathsf{X}$,
    $$t\restriction \mathsf{X} = s \restriction\mathsf{X}\text{ implies }t\vDash \phi \iff s\vDash \phi,$$
    where for teams $u$ we write $u\restriction \mathsf{X}:=\{v\restriction \mathsf{X}\mid v\in u \}$.
\end{definition}
This is a nontrivial property---see, e.g., \cite{galliani2012,durand2018multiteam,yang2022} for discussion of logics with formulas which may fail to be local in this sense. While for most of the connectives considered it is clear that they respect locality, this is less easy to see in the case of the disjunctions $\veedot$ and $\vveedot$, which quantify over superteams (and in fact, as we show in Section \ref{c:section:convex_modal}, in a non-downward-closed setting these connectives bring about the failure of ($0$-)bisimulation invariance, a closely related property).

\begin{proposition} \label{c:prop:locality}
    Formulas of $\CONDEP$, $\CONINQ$, and $\PLIM$ are local.
\end{proposition}
\begin{proof}
    Proved by induction on the complexity of formulas. We will only explicitly show that $\veedot$ preserves locality; the proof for $\vveedot $ is similar. Let $\phi \veedot\psi(\mathsf{X})$ be such that $\phi$ and $\psi$ are local, and let $t$ over $\mathsf{Y}\supseteq \mathsf{X}$ and $s$ over $\mathsf{Z}\supseteq \mathsf{X}$ be such that for $t':=t\restriction \mathsf{X}$ and $s':=s \restriction \mathsf{X}$, we have $t'=s'$.

    Suppose $t\vDash \phi \veedot \psi$. Then $t\subseteq u_\phi \cup u_\psi$, where $u_\phi \vDash \phi$ and $u_\psi\vDash \psi$. Let $u'_\phi:=u_\phi\restriction \mathsf{X}$ and $u_\psi':=u_\psi \restriction\mathsf{X}$. Then clearly $t'\subseteq u'_\phi \cup u'_\psi$, whence also $s'\subseteq u'_\phi \cup u'_\psi$.
    
    Let $u^*_\phi:=\{v\in 2^\mathsf{Z}\mid v\restriction \mathsf{X}\in u'_\phi \}$ and similarly for $u^*_\psi$; i.e., $u^*_\phi$ is the maximal team over $\mathsf{Z}$ whose restriction to $\mathsf{X}$ is $u'_\phi$. Given $s'\subseteq u'_\phi \cup u'_\psi$, we have $s\subseteq u^*_\phi \cup u^*_\psi $. By $u_\phi \restriction \mathsf{X}= u'_\phi=u^*_\phi\restriction \mathsf{X}$ and $u_\phi \vDash \phi$ and the locality of $\phi$, we have $u^*_\phi \vDash \phi$; similarly, $u^*_\psi \vDash \psi$. Therefore, $s\vDash \phi \veedot \psi$. The other direction is symmetrical.
\end{proof}

\subsubsection{Preliminaries: Expressive Completeness}

We measure the expressive power of  logics in terms of the properties (or denotations or meanings) expressible in them, defined as follows.

\begin{definition}[Team properties and expressive completeness] \label{c:definition:expressive_completeness}
A \emph{(propositional team) property} over a set of propositional letters $\mathsf{X}\subseteq\mathsf{P}$ is a class of (propositional) teams over $\mathsf{X}$. For each formula $\phi(\mathsf{X})$, we denote by $\left\Vert\phi\right\Vert_{\mathsf{X}}$ (or simply $\left\Vert\phi\right\Vert$) the property over $\mathsf{X}$ \emph{expressed} by $\phi$:
$$\left\Vert\phi\right\Vert_{\mathsf{X}}:=\{t\subseteq 2^\mathsf{X}\mid t\vDash \phi\}.$$
Given a logic $\LOGIC$ we write $\left \Vert\LOGIC\right \Vert$ for the class of all properties expressed by some formula $\phi$ of $\LOGIC$, i.e.,
    $$\left\Vert\LOGIC\right\Vert:=\{\left\Vert \phi \right\Vert_{\mathsf{X}}\mid \phi(\mathsf{X})\text{ is a formula of }\LOGIC\},$$
and given a class of properties $\mathbb{P}$, we write $\mathbb{P}^{<\infty}$ for the class of all properties in $ \mathbb{P}$ that are over some finite $\mathsf{X}\subseteq \mathsf{P}$. We say that a logic $\LOGIC$ is \emph{expressively complete} for a class of properties $\mathbb{P}$ if $\left\Vert \LOGIC\right \Vert^{<\infty}=\mathbb{P}^{<\infty}$. That is, ($\subseteq$) each property $\left\Vert\phi\right\Vert$ definable by a formula $\phi$ of $\LOGIC$ is in $\mathbb{P}$ (whence it is also in $\mathbb{P}^{<\infty}$), and ($\supseteq$) each property in $\mathbb{P}^{<\infty}$ is definable by a formula of $\LOGIC$. When discussing expressive completeness theorems, we will often write simply $\mathbb{P}$ for $\mathbb{P}^{<\infty}$.
\end{definition}

The definition of closure properties is extended to team properties in the obvious way. For example, a property $\PPP$ is downward closed if $t\in \PPP$ and $s\subseteq t$ imply $s\in \PPP$. We write $\mathbb{C}$/$\mathbb{CU}$/$\mathbb{D}$/$\mathbb{DE}$/$\mathbb{U}$/$\mathbb{F}$ for the class of all convex/convex and union-closed/downward-closed/downward-closed and empty-team-closed ($\emptyset\in \PPP$)/upward-closed/flat properties, respectively.

\subsubsection{Expressive Completeness for Convex Properties} With these preliminaries in place, we can proceed with the proof of our first main theorem: that each of $\CONDEP$, $\CONINQ$, and $\PLIM$ is complete for $\mathbb{C}$, i.e.,
    $$\left\Vert \CONDEP\right \Vert=\left\Vert \CONINQ\right \Vert=\left\Vert \PLIM\right \Vert=\mathbb{C}.$$ 
Note that we have already shown, in Proposition \ref{c:prop:convex_logics_convex}, that $\left\Vert \CONDEP\right \Vert\subseteq \mathbb{C}$, etc.

To prove the other inclusion, we construct, for each property $\PPP\in \mathbb{C}^{<\infty}$, a formula in $\left\Vert\CONDEP\right\Vert$ that expresses $\PPP$---a characteristic formula for $\PPP$; and similarly for the other logics. 

First, we note that the empty property $\PPP=\emptyset$ is convex, and that it is expressible in each of our logics using the formula(s) $\Bot$.


For nonempty properties, we will need to recall characteristic formulas both for valuations and for teams $\chi^\mathsf{X}_v,\chi^\mathsf{X}_t\in \PLV/\PLD/\PLI$ (see, e.g., \cite{CiardelliRoelofsen2011,hella2014,yangvaananen2016,yang2017}). For the remainder of this section, to avoid repetition, we fix a finite set of propositional letters $\mathsf{X}\subseteq \mathsf{P}$ and assume all properties are over $\mathsf{X}$; it will be clear how to parametrize the results to obtain ones that hold independently of any given $\mathsf{X}$.

For a valuation $v$, define the characteristic formula
$$\chi^\mathsf{X}_v:=\bigwedge\{p\mid p\in \mathsf{X},v(p)=1\}\land \bigwedge\{\lnot p\mid p\in \mathsf{X},v(p)=0\}.$$
It is then easy to see that
$$w\vDash \chi^\mathsf{X}_v \iff w\restriction \mathsf{X}=v\restriction \mathsf{X} ,$$ and if $w,v\in 2^\mathsf{X}$, then $w\vDash \chi^\mathsf{X}_v \iff w=v$. We usually write simply $\chi_v$. 

For a team $t$, we, in turn, define the characteristic formula
$$ \chi^\mathsf{X}_t:= \bigcurlyvee_{v\in t} \chi^\mathsf{X}_v.$$
Then:
$$s\vDash \chi^\mathsf{X}_t\iff s\restriction \mathsf{X}\subseteq t\restriction\mathsf{X}  \quad \quad \text{(where, we recall, for a team $u$, $u\restriction \mathsf{X}:=\{v\restriction \mathsf{X}\mid v\in u\}$),}$$
and if $t,s\subseteq 2^\mathsf{X}$, then $s\vDash \chi^\mathsf{X}_t \iff s\subseteq t$. Again, we usually write simply $\chi_t$. (Note that since $\mathsf{X}$ is finite, the conjunctions in $\chi_v^\mathsf{X}$ are finite, whence these formulas are well-defined. Similarly, there are only finitely many distinct formulas $\chi_v^\mathsf{X}$, so the disjunction in $\chi_t^\mathsf{X}$ is finite.)

It is instructive to present our construction of the characteristic formulas for convex properties schematically, in a way that sets out a convex set of teams as the intersection of a down-set with an up-set. Accordingly, let us say that a formula schema $\chi_\QQQ$ (yielding a formula $\chi_\PPP=\chi_{\PPP/\QQQ}$ for each property $\PPP$) is a \emph{$D$-schema} if for any nonempty property $\PPP$ and team $t$ over $\mathsf{X}$,
\begin{align*}
 t\vDash \chi_{\PPP}  \iff \exists s\in \PPP: t\subseteq s;
\end{align*}
and that $\chi_\QQQ$ is a \emph{$U$-schema} if for any such $\PPP$ and $t$,
\begin{align*} 
 t\vDash \chi_{\PPP}  \iff \exists s\in \PPP: t\supseteq s.
\end{align*}

Note that a $D$-schema will yield characteristic formulas for nonempty downward-closed properties $\PPP$: for any such schema $\chi_\QQQ$ and property $\PPP$, $\left\Vert\chi_\PPP\right\Vert=\PPP$. Similarly, a $U$-schema yields characteristic formulas for nonempty upward-closed properties. The conjunction of a $D$-schema and a $U$-schema will then yield characteristic formulas for nonempty convex properties:

\begin{lemma} \label{c:lemma:convex_char_formula} 
Let $\chi^{\mathrm{D}}_\QQQ$ be a $D$-schema and $\chi^{\mathrm{U}}_\QQQ$ be a $U$-schema. Then for any nonempty convex $\PPP$, $\left\Vert \chi^{\mathrm{D}}_\PPP\land \chi^{\mathrm{U}}_\PPP\right \Vert=\PPP  $.
\end{lemma}
\begin{proof}
$\supseteq:$ For any $t\in \PPP$, $t\subseteq t\subseteq t$, whence $t\vDash \chi^{\mathrm{D}}_\PPP\land \chi^{\mathrm{U}}_\PPP$.

$\subseteq:$ If $t\vDash \chi^{\mathrm{D}}_\PPP\land \chi^{\mathrm{U}}_\PPP$, then for some $s,u\in \PPP$, $s\subseteq t\subseteq u$, whence $t\in \PPP$ by convexity.
\end{proof}
It therefore suffices to define a $D$-schema $\chi^{\mathrm{D}}_\QQQ$ and a $U$-schema  $\chi^{\mathrm{U}}_\QQQ$ in each of our logics. We begin with the $U$-schemas. These can be constructed in a similar manner irrespective of the logic:
\begin{lemma} \label{c:lemma:upward_char_formula}
For each $\PPP=\{t_1,\ldots,t_n\}\neq \emptyset$, let  
    $$\chi^{\mathrm{U}}_\PPP:=\bigwedge_{v_1\in t_1,\ldots,v_n\in t_n} \filleddiamond (\chi_{v_1}\curlyvee \cdots \curlyvee \chi_{v_n}).$$
Then for any $t\subseteq 2^{\mathsf{X}}$, $t\vDash \chi^{\mathrm{U}}_\PPP  \iff \exists t_i\in \PPP: t\supseteq t_i$.
\end{lemma}
\begin{proof}

$\Longleftarrow$: Let $t_i\in \PPP$ be such that $t\supseteq t_i$. For each $v_i\in t_i$, we have that $v_i\vDash \chi_{v_i}$, so that also $v_i\vDash \chi_{v_1}\curlyvee \cdots \curlyvee \chi_{v_n}$ for any $v_1\in t_1,\ldots, v_{i-1}\in t_{i-1},v_{i+1}\in t_{i+1},\ldots, v_{n}\in t_{n}$. Therefore, since $v_i\in t_i\subseteq t$, $t\vDash \filleddiamond(  \chi_{v_1}\curlyvee \cdots \curlyvee \chi_{v_n})$. Repeating, this argument, we have $t\vDash \bigwedge_{v_1\in t_1,\ldots,v_n\in t_n} \filleddiamond (\chi_{v_1}\curlyvee \cdots \curlyvee \chi_{v_n})$.

\sloppy
$\Longrightarrow$: Let $t\vDash \bigwedge_{v_1\in t_1,\ldots,v_n\in t_n} \filleddiamond (\chi_{v_1}\curlyvee \cdots \curlyvee \chi_{v_n})$. If $\PPP$ has the empty team property, we have $\emptyset\subseteq t$ where $\emptyset\in \PPP$; we may therefore assume $\PPP$ does not have the empty team property, whence $t_i\neq \emptyset$ for all $1\leq i\leq n$, so $\bigwedge_{v_1\in t_1,\ldots,v_n\in t_n} \filleddiamond (\chi_{v_1}\curlyvee \cdots \curlyvee \chi_{v_n})$ is not simply $\top$ (recall that $\bigwedge \emptyset=\top$). Assume for contradiction that $t_i\not\subseteq t$ for all $t_i\in \PPP$. Then for each $t_i\in \PPP$ there is some $w_i\in t_i$ such that $w_i\notin t$. By $t\vDash \bigwedge_{v_1\in t_1,\ldots,v_n\in t_n} \filleddiamond (\chi_{v_1}\curlyvee \ldots \curlyvee \chi_{v_n})$, there is, for each $v_1\in t_1,\ldots, v_n\in t_n$, a nonempty $t_{v_1,\ldots, v_n}\subseteq t$ such that $t_{v_1,\ldots,v_n}\vDash \chi_{v_1}\curlyvee \cdots \curlyvee \chi_{v_n}$. Therefore, in particular, there is a nonempty $t_{w_1,\ldots, w_n}\subseteq t$ such that $t_{w_1,\ldots,w_n}\vDash \chi_{w_1}\curlyvee \cdots \curlyvee \chi_{w_n}$. Then $t_{w_1,\ldots,w_n}\subseteq \bigcup_{1\leq i \leq n}\{w_i\}$ and $t_{w_1,\ldots,w_n}\neq \emptyset$. Given $t_{w_1,\ldots,w_n}\subseteq t$ we have $t\cap \bigcup_{1\leq i \leq n}\{w_i\}\neq \emptyset$, contradicting the fact that $w_i\notin t$ for all $t_i\in \PPP$.
\end{proof}
\fussy

We now turn to the $D$-schemas. We will construct these in a distinct manner in each of the three logics $\CONDEP$, $\CONINQ$ and $\PLIM$, drawing on well-known schemas for dependence logic $\DEP$ and inquisitive logic $\INQ$ that  have been used to prove the following:

 \begin{theorem}[\cite{CiardelliRoelofsen2011,yangvaananen2016}] \label{c:theorem:dep_inq_expressive_completeness} Each of $\DEP$ and $\INQ$ is expressively complete for the class of all downward-closed properties with the empty team property: $$\left\Vert \DEP \right\Vert=\left\Vert\INQ\right \Vert=\mathbb{DE}.$$
    \end{theorem}

We begin with the logics $\CONINQ$ and $\PLIM$. For these, we make use of the schema $\bigvvee_{s\in \QQQ}\chi_s$ of inquisitive logic $\INQ$. It is easy to see that this is already a $D$-schema; it thus suffices to find equivalent schemas of $\CONINQ$ and $\PLIM$. In the former case this is trivial given that for any downward-closed (and hence for any classical) $\phi$ and $\psi$, we have $\phi\vvee\psi\equiv \phi\vveedot \psi$.
\begin{lemma}\label{c:lemma:coninq_DE_formulas}
For each  $\PPP\neq \emptyset$, let $$\chi^{\mathrm{D}}_\PPP:=\bigvveedotds_{s\in \PPP}\chi_s.$$
Then for any $t\subseteq 2^{\mathsf{X}}$, $t\vDash \chi^{\mathrm{D}}_\PPP  \iff \exists s\in \PPP: t\subseteq s$.  
\end{lemma}
\begin{proof}
    Immediate from the definitions and the fact that $s\vDash \chi_s$.
\end{proof}

As for $\PLIM$, we show below that there is a schema equivalent to $\bigvvee_{s\in \QQQ}\chi_s$ in $\PLIM$ by showing that the inquisitive disjunction of any collection of flat formulas is definable in $\PLIM$. (Thus showing that the disjunction $\vveedot$ and the schema above are not in fact required to prove the expressive completeness of $\CONINQ$. We discuss our reasons for introducing $\vveedot$ and the logic $\CONINQ$ in Section \ref{c:section:uniform_definability}.)

\begin{proposition}
\label{c:prop:global_disjunction_definable}
For any nonempty finite collection $\{\phi_i\}_{i\in I}$ of flat formulas and any $j\in I$,
    $$\left(\bigwedge_{i\in I\setminus\{j\}}\filleddiamond \lnot\phi_i\right) \to \phi_j\equiv \bigvveeds_{i\in I}\phi_i\equiv \bigvveedotds_{i\in I}\phi_i.$$
\end{proposition}

To prove this, we will use that the intuitionistic implication is material when the antecedent is upward closed and the consequent downward closed:
\begin{lemma}
    Let $\varphi_u$ be upward closed, and $\varphi_d$ downward closed. Then $t\vDash \varphi_u\to \varphi_d$ iff $t\nvDash \varphi_u$ or $t\vDash \varphi_d$.
\end{lemma}
\begin{proof}
    For the left-to-right direction, suppose for contraposition that $t\vDash \varphi_u$ and $t\nvDash \varphi_d$. Then clearly $t\nvDash \varphi_u\to \varphi_d$.

    For the right-to-left direction, if $t\nvDash \varphi_u$, then $s\nvDash \varphi_u$ for all $s\subseteq t$, hence vacuously $t\vDash \varphi_u\to \varphi_d$; and if $t\vDash \varphi_d$, then $s\vDash \varphi_d$ for all $s\subseteq t$, hence $t\vDash \varphi_u\to \varphi_d$.
\end{proof}
\begin{proof}[Proof of Proposition \ref{c:prop:global_disjunction_definable}]
    Observe that $\bigwedge_{i\in I\setminus\{j\}}\filleddiamond \lnot\phi_i$ is upward closed, and $\phi_j$ is downward closed. Thus, 
\begin{align*}
    t\vDash (\bigwedge_{i\in I\setminus\{j\}}\filleddiamond \lnot\phi_i) \to \phi_j && \text{iff} && t\nvDash (\bigwedge_{i\in I\setminus\{j\}}\filleddiamond \lnot\phi_i) \textit{ or } t\vDash \phi_j\\
    && \text{iff} && \text{there is some } i\in I\setminus\{j\} \text{ s.t. } t\nvDash \filleddiamond \lnot\phi_i \textit{ or } t\vDash \phi_j\\
    && \overset{(a)}{\text{iff}} && \text{there is some } i\in I\setminus\{j\} \text{ s.t. } t\vDash \phi_i \textit{ or } t\vDash \phi_j\\
    && \text{iff} && t\vDash \bigvveeds_{i\in I}\phi_i,
\end{align*}
where we in $(a)$ used that for flat formulas $\varphi_f$: $t\vDash \varphi_f $ iff $t\nvDash \filleddiamond\neg \varphi_f $.
\end{proof}

It remains only to show that we can define a $D$-schema in $\CONDEP$. We make use of the schema of dependence logic $\DEP$ used to prove Theorem \ref{c:theorem:dep_inq_expressive_completeness}. This is defined by first letting $\gamma^{\mathsf{X}}_0:=\bot$, $\gamma^{\mathsf{X}}_1:= \bigwedge_{p \in \mathsf{X}}\con{p}$, and for $n\geq 2$, $\gamma^{\mathsf{X}}_n:=\bigvee_{1\leq i\leq n} \gamma_1^{\mathsf{X}}$. Then it is easy to see that for $t\subseteq 2^{\mathsf{X}}$, we have  $t\vDash \gamma^{\mathsf{X}}_n\iff |t|\leq n$, where $|t|$ is the cardinality of $t$. One then defines, for each nonempty $t\subseteq 2^\mathsf{X}$, $\xi^{\mathsf{X}}_t:=\gamma^{\mathsf{X}}_{|t|-1}\vee \chi^{\mathsf{X}}_{2^{\mathsf{X}}\setminus t}$. It can then be shown that for $t\subseteq 2^{\mathsf{X}}$, $t\vDash \xi^{\mathsf{X}}_s \iff s\not\subseteq t$. Finally, the schema is given by $\bigwedge_{s\in\left\Vert\top \right\Vert\setminus \QQQ } \xi^{\mathsf{X}}_s$: for any $\PPP\in \mathbb{DE}$, $\PPP=\left\Vert \bigwedge_{s\in\left\Vert\top \right\Vert\setminus \PPP } \xi^{\mathsf{X}}_s\right\Vert$. Now, it can be verified that this schema is \emph{not} a $D$-schema---the schema works as desired for nonempty downward-closed properties, but if $\PPP$ is not downward closed, it need not be the case that $\bigwedge_{s\in\left\Vert\top \right\Vert\setminus \PPP } \xi^{\mathsf{X}}_s \iff \exists s\in \PPP:t\subseteq s$ (in particular, the schema does not function as desired for all nonempty convex properties). Given the properties of the formulas $\xi_s$, we do, however, have:

\begin{lemma}[\cite{yangvaananen2016}] \label{c:lemma:dep_char_formula_prop}
 For $\PPP$ such that $\emptyset\notin \PPP$, and $t\subseteq 2^{\mathsf{X}}$: $$t\vDash\bigwedge_{s\in  \PPP } \xi^{\mathsf{X}}_s\iff s\not\subseteq t \text{ for all }s\in \PPP.$$
\end{lemma}
We use this to construct our $D$-schema:
\begin{lemma} \label{c:lemma:CONDEP_DE_formulas}
    For each $\PPP\neq \emptyset$, let $$\chi^{\mathrm{D}}_\PPP:=\bigwedge_{u\in  \RRR } \xi'^{\mathsf{X}}_u,$$
    where $\xi'^{\mathsf{X}}_u\in \CONDEP$ is defined by replacing each $\vee$ in $\xi^{\mathsf{X}}_u$ with $\veedot$, and we let $\RRR:=\{u\subseteq 2^{\mathsf{X}}\mid u\not \subseteq s$ for all $s\in \PPP\}$. Then for any $t\subseteq 2^{\mathsf{X}}$, $t\vDash \chi^{\mathrm{D}}_\PPP \iff \exists s\in \PPP:t\subseteq s$.
\end{lemma}
\begin{proof}
Note that for all downward-closed formulas $\phi$ and $\psi$, $\phi\vee\psi\equiv \phi\veedot \psi$, and that, since $\PPP\neq \emptyset$, we have $\emptyset\notin \RRR$. Therefore, by Lemma \ref{c:lemma:dep_char_formula_prop}, we have that $t\vDash\chi^{\mathrm{D}}_\PPP\iff$  [$u\not\subseteq t$ for all $u\in \RRR$]. We show that [$u\not\subseteq t$ for all $u\in \RRR$] $\iff \exists s\in \PPP:t\subseteq s$. 

    $\Longrightarrow$: For contraposition, assume that there is no $s\in \PPP$ such that $t\subseteq s$. Then $t\in \RRR$, whence there is $u\in \RRR$ such that $u\subseteq t$, namely $u=t$.

    $\Longleftarrow$: For contraposition, assume that there is some $u\in \RRR$ such that $u\subseteq t$, and let $s\in \PPP$ be arbitrary. Then $t\not\subseteq s$, because otherwise $u\subseteq t\subseteq s$, contradicting $u\in\RRR$.
\end{proof}

And we are done:

\begin{theorem}\label{c:theorem:convex_expressive_completeness}
    Each of $\CONDEP$, $\CONINQ$, and $\PLIM$ is expressively complete for the class of all convex properties: $$\left\Vert \CONDEP\right\Vert=\left\Vert \CONINQ\right\Vert=\left\Vert \PLIM\right \Vert=\mathbb{C}.$$
\end{theorem}
\begin{proof}
The direction $\left\Vert\LOGIC\right\Vert\subseteq \mathbb{C}$ is by Proposition \ref{c:prop:convex_logics_convex} for each relevant $\LOGIC$. For the direction $\mathbb{C}\subseteq \left\Vert\LOGIC\right\Vert$, let $\PPP\in \mathbb{C}$, with $\PPP$ over $\mathsf{X}$. If $\PPP=\emptyset$, then $\PPP=\left\Vert\Bot\right\Vert_\mathsf{X}\in \left\Vert \LOGIC \right\Vert$. If $\PPP\neq \emptyset$, the result follows by Lemmas/Propositions \ref{c:lemma:convex_char_formula}, \ref{c:lemma:upward_char_formula}, \ref{c:lemma:coninq_DE_formulas}, \ref{c:prop:global_disjunction_definable}, and \ref{c:lemma:CONDEP_DE_formulas}.
\end{proof}

Next, we consider similar questions in the context of modal team semantics.

\subsection{Modal Properties}
\label{c:section:convex_modal}

We define two distinct modal extensions of each of our classical bases. One type of extension uses the \emph{flat} modalities $\Diamond$ and $\Box$,  known from inquisitive  modal logic \cite{ciardelli2016} and Bilateral State-Based Modal Logic \cite{aloni2022}; the other uses the modalities $\diamonddiamond$ and $\boxbox$  from modal dependence logic and related logics \cite{vaananen2008,hella2015,kontinen20142}. For $i\in \{\vee, \veedot, \to\}$ and $j\in \{\Diamond, \diamonddiamond\}$, we call the basis $\ML_{i,j}$ \emph{classical modal logic}, and define it in the obvious way. The extensions $\ML_{\veedot,j}(\con{\cdot},\filleddiamond)$, $\ML_{\to,j}(\vveedot,\filleddiamond)$ and $\ML_{\to,j}(\filleddiamond)$, where $j\in \{\Diamond, \diamonddiamond\}$ are likewise defined in the obvious way.

The main result of this section is a modal analogue of our expressive completeness theorems. We show that the logic $\LPLIM$ is complete for convex team properties invariant under bounded team bisimulation. We also show that each of the other modal extensions is, in contrast, either not convex or not team bisimulation-invariant.

We proceed with a few technical preliminaries. First, modal team logics are interpreted on teams in standard Kripke models, defined as follows. 

\begin{definition} \label{c:definition:models}
A \emph{(Kripke) model} (over $\mathsf{X}\subseteq \mathsf{P}$) is a triple $M = (W,R,V)$, where
\begin{itemize}
    \item[--] $W$ is a nonempty set, whose elements are called \emph{(possible) worlds};
    \item[--] $R \subseteq W \times W$ is a binary relation, called the \emph{accessibility relation};
    \item[--] 
    $V: \mathsf{X} \to \wp(W)$
    is a function, called the \emph{valuation}.
\end{itemize}  

We call a subset $t\subseteq W$ of $W$ a (modal) \emph{team} on $M$.
\end{definition} 
For any world $w$ in $M$, define, as usual, $R[w]:=\{v \in W\mid wRv\}$. Similarly, for any team $t$ on $M$, define $R[t]:=\bigcup_{w\in t}R[w]$ and $R^{-1}[t]:=\{v \in W\mid \exists w\in t:vRw\}$. We write $tRs$ and say that $s$ is a \emph{successor team} of $t$ if $s\subseteq R[t]$ and $t\subseteq R^{-1}[s]$.

Most semantic clauses are as in the propositional case (for example, a team $t$ on $M$ makes $p$ true---written $M,t\vDash p$---just in case $t\subseteq V(p)$). We only explicitly give the clauses for the modalities.
\begin{align*}
    &M,t\vDash \Diamond \phi &&:\iff &&\forall w\in t: \exists s\subseteq R[w]: s\neq \emptyset \text{ and }M,s\vDash \phi;\\
    &M,t\vDash \Box \phi &&:\iff &&\forall w\in t: M,R[w] \vDash \phi;\\
    &M,t\vDash \diamonddiamond \phi &&:\iff &&\exists s\subseteq W: tRs  \text{ and }M,s\vDash \phi;\\
    &M,t\vDash \boxbox \phi &&:\iff &&M,R[t]\vDash \phi.
\end{align*}

Closure properties are defined as before (recall Definition \ref{c:definition:closure_props}), and it is easy to see that Facts \ref{c:fact:closure_props} and \ref{c:fact:classical_correspondence} carry over to the modal setting. Observe that $\Diamond\phi$ and $\Box \phi$ are flat for all $\phi$; given this, an easy extension of Proposition \ref{c:prop:convex_logics_convex} yields:

\begin{proposition} \label{c:prop:modal_flat_convex_logics_convex}
    Formulas of $\LCONDEP$, $\LCONINQ$ and $\LPLIM$ are convex.
\end{proposition}

The extensions with the modalities $\diamonddiamond$ and $\boxbox$ are, however, not convex, as the fact below shows. The fact further shows that, as with $\vee$ and $\vvee$, the diamond $\diamonddiamond$ does not preserve convexity and so no logic with $\diamonddiamond$ can be complete for the class of all convex modal properties.

\begin{fact}\label{c:fact:global_disjunction_convexity_break} There are (i) $\psi\in \GPLIM$ and (ii) $\chi\in \GCONDEP$ that are not convex. There is (iii) a convex $\phi$ such that $\diamonddiamond \phi$ is not convex.
   \end{fact}
   \begin{proof}
    Consider the formula $\phi:= ((\filleddiamond p \land \filleddiamond \lnot p)\to \bot) \land \filleddiamond r$ and note that $\phi\equiv \,\con{p}\land \filleddiamond r$. It is easy to see that $\phi$ is convex; to show (i--iii) it therefore suffices to show that $\diamonddiamond \phi$ is not convex. Consider the following model $M=(W,R,V)$ (with $R$ represented using arrows):
   
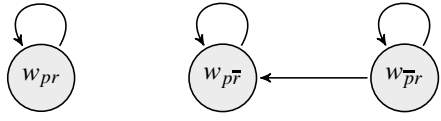
\begin{center}
   \begin{tikzpicture}[modalx]
   \node[world] (pr) {$w_{pr}$};
   \node[world] (pnor) [right=of pr] {$w_{p\overline{r}}$};
   \node[world] (nopr) [right=of pnor]  {$w_{\overline{p}r}$};
   \path[->] (nopr) edge (pnor);
   \path[->] (pr) edge[reflexive] (pr);
   \path[->] (nopr) edge[reflexive] (nopr);
   \path[->] (pnor) edge[reflexive] (pnor);
   \end{tikzpicture}
   \end{center}
   We have $M,\{w_{\overline{p}r}\}\vDash \phi$ and $\{w_{\overline{p}r}\}R\{w_{\overline{p}r}\}$, whence $M,\{w_{\overline{p}r}\}\vDash \diamonddiamond\phi$. We also have $M,\{w_{pr}, w_{p\overline{r}}\}\vDash \phi$ and $\{w_{pr}, w_{p\overline{r}}, w_{\overline{p}r}\}R\{w_{pr}, w_{p\overline{r}}\}$, whence $M,\{w_{pr}, w_{p\overline{r}}, w_{\overline{p}r}\} \vDash \diamonddiamond \phi$. But we have $M,\{ w_{p\overline{r}}, w_{\overline{p}r}\}\nvDash \diamonddiamond \phi$. For the only $t\subseteq W$ such that $\{ w_{p\overline{r}}, w_{\overline{p}r}\}Rt$ are $\{ w_{p\overline{r}}, w_{\overline{p}r}\}$ and $\{w_{p\overline{r}}\}$, and for neither of these do we have $t\vDash\phi$.
   \end{proof}

    One might consider introducing a downward-closed variant of $\diamonddiamond$, similar to the ideal disjunction $\veedot$ and the downjunction $\vveedot$, to force convexity. However, we show directly below that the disjunctions $\veedot$ and $\vveedot$ (in combination with connectives which break downward closure) bring about the failure of team bisimulation-invariance. The most obvious downward-closed variant of $\diamonddiamond$ would similarly violate bisimulation-invariance in a non-downward-closed setting.

    The expressive power of the basic modal language can be characterized using bisimulation relations: it corresponds to the bisimulation-invariant fragment of classical first-order logic \cite{vanbenthem,blackburn2001}. Recall that two pointed models (model-world pairs) $(M,w)$ and $(M',w')$ (where $M=(W,R,V)$ and $M'=(W',R',V')$) are \emph{bisimilar} (notation: $M,w\leftrightharpoons M',w'$) if there is a relation $Z\subseteq W\times W'$ such that $(w,w')\in Z$ and for all $(x,x')\in Z$:
    \begin{itemize}
        \item[-] (base) $x\in V(p) $ iff $x'\in V'(p) $, for all propositional letters $p$;
        \item[-] (forth) $xRy$ implies $x'R'y'$ for some $y'\in W'$ with $yZy'$;
        \item[-] (back) $x'R'y'$ implies $xRy$ for some $y\in W$ with $yZy'$.
    \end{itemize}
    
    In a similar manner, most team-based modal logics correspond to \emph{team bisimulation}-invariant fragments of first-order logic \cite{kontinen2014}, where two pointed team models (model-team pairs) $(M,t)$ and $(M',t')$ are \emph{team bisimilar} (notation: $M,t\leftrightharpoons M',t'$) if
    \begin{itemize}
        \item[-] (forth) for each $w\in t$, there is a $w'\in t'$ such that $M,w\leftrightharpoons M',w'$;
        \item[-] (back) for each $w'\in t'$, there is a $w\in t$ such that $M,w\leftrightharpoons M',w'$.
    \end{itemize}

    An easy induction shows that formulas $\phi$ of $\LPLIM$ and $\GPLIM$ are team bisimulation invariant: if $M,t\leftrightharpoons M',t'$, then $M,t\vDash \phi$ iff $M',t'\vDash \phi$. By contrast, we have:
    \begin{fact}\label{c:fact:failure_of_bisim_invariance}
        Formulas of the logics $\LCONDEP$, $\GCONDEP$, $\LCONINQ$ and $\GCONINQ$ need not be bisimulation invariant.
    \end{fact}
    \begin{proof}
        Let $M=(\{w,v\},R, V)$ and $M'=(\{w\},R', V')$ be models such that $R=R'=\emptyset$, $V(p)=\{v\}$ and $V'(p)=\emptyset$. Then $M,\{w\}\leftrightharpoons M',\{w\}$, but we have $M,w\vDash \filleddiamond  p \veedot \filleddiamond  p$ and $M,w\vDash \filleddiamond  p \vveedot \filleddiamond  p$, whereas $M',w\nvDash \filleddiamond  p \veedot \filleddiamond p$ and $M',w\nvDash \filleddiamond  p \vveedot \filleddiamond  p$.
    \end{proof}

    One can define $n$-team bisimulations ($n\in \mathbb{N}$) by replacing full world-bisimilarity by $n$-world bisimilarity in the definition of team bisimulation; the above shows that formulas of these logics need not even be $0$-bisimulation invariant (even if they are local, cf. Proposition~\ref{c:prop:locality}). Note also that in the proof above, we used the fact that the \emph{universal diamond} modality $E$ (see, e.g., \cite{gorankopassy}) is definable in the logics considered ($E\alpha \equiv\filleddiamond \alpha \veedot\filleddiamond \alpha \equiv \filleddiamond \alpha \vveedot\filleddiamond \alpha$).

We move on to the modal analogue of Theorem \ref{c:theorem:convex_expressive_completeness} for the logic $\LPLIM$. We omit the details: given Proposition \ref{c:prop:modal_flat_convex_logics_convex} and bisimulation invariance, the proof of this result is completely analogous to that for $\PLIM$. One can define natural analogues of team properties, expressive completeness, and the formulas $\chi^{\mathsf{X}}_v$ and $\chi^{\mathsf{X}}_t$ in the modal setting (see \cite{hella2014,kontinen2014} for details). A modal version of the proof of Theorem \ref{c:theorem:convex_expressive_completeness} then gives:

\begin{theorem}\label{c:theorem:convex_modal_expressive_completeness}
    $\LPLIM$ is expressively complete for the class of all convex team modal properties invariant under bounded bisimulation.
\end{theorem}

Motivated by Fact~\ref{c:fact:global_disjunction_convexity_break} and~\ref{c:fact:failure_of_bisim_invariance}, it would be of interest to characterize the expressive powers of $\LCONDEP$, $\GCONDEP$, $\LCONINQ$, $\GCONINQ$, $\GPLIM$. We leave these as open questions for future work.

\section{Convex and Union-Closed Properties}
\label{c:section:union-closed_convex}

We now turn to properties that are not only convex, but also union closed. In Section \ref{c:section:union-closed_convex_prop}, we show that the logic $\PLNE$ \cite{yang2017} is expressively complete with respect to the class of all convex and union-closed propositional properties. After that, in Section \ref{c:section:union-closed_convex_modal}, we show modal analogues of this result for, notably, Aloni's Bilateral State-Based Modal ($\BSML$) \cite{aloni2022}, as well as two modal extensions of $\PLNE$.

\subsection{Propositional Properties}
\label{c:section:union-closed_convex_prop}
 $\PLNE$ was first considered in~\cite{yang2017} and is defined by extending the classical basis $\PLV$ with the \emph{nonemptiness atom} $\NE$, true in a team just in case the team is nonempty:
 $$t\vDash \NE \quad\iff\quad t\neq \emptyset.$$
We define $\top$, $\alpha \curlyvee \beta$, $\bigvee\emptyset$ and $\bigwedge\emptyset$ as before, and let $\Bot:=\bot \land \NE$ and $\filleddiamond \phi:=(\phi\land \NE)\vee \top$. It is easy to see that $\Bot$ and $\filleddiamond \phi$ as defined here are equivalent with their counterparts as defined in Section \ref{c:section:convex_prop}.

$\PLNE$ clearly violates downward closure and the empty team property, but we do have union closure and convexity---interestingly, in contrast with Fact \ref{c:fact:disjunctions_convexity_break}, the split disjunction $\vee$ does preserve convexity in a convex and union-closed setting.
\begin{proposition} \label{c:prop:PLNE_convex_union_closed}
    Formulas of $\PLNE$ are convex and union closed.
\end{proposition}
\begin{proof}
    Proved by induction on the structure of formulas. Most cases are straightforward; note in particular that $\NE$ is upward closed and therefore also convex.
    
    We only explicitly show that $\phi \vee\psi$ is convex whenever $\phi$ and $\psi$ are convex and union closed. Let $s\vDash \phi \vee\psi$, $t\vDash \phi \vee\psi$, and $s\subseteq u \subseteq t$. Then $s=s_\phi\cup s_\psi$ and $t=t_\phi\cup t_\psi$ where $t_\phi\vDash \phi$, etc. Let $u_\phi:=(s_\phi \cup t_\phi)\cap u$ and $u_\psi:=(s_\psi \cup t_\psi)\cap u$. We have $s_\phi \cup t_\phi\vDash \phi$ and $s_\psi \cup t_\psi\vDash \psi$ by union closure, whence $u_\phi\vDash \phi$ by convexity since $s_\phi \subseteq u_\phi \subseteq s_\phi \cup t_\phi$. Similarly, $u_\psi\vDash \psi$. Clearly $u=u_\phi \cup u_\psi$, whence $u\vDash \phi \vee\psi$.
\end{proof}

This establishes that $\left\Vert \PLNE\right\Vert\subseteq \mathbb{CU}$. We now further show---solving a problem that was left open in \cite{yang2017}---that $\left\Vert \PLNE\right\Vert= \mathbb{CU}$, viz. that $\PLNE$ is expressively complete for $\mathbb{CU}$.

We prove this converse inclusion ($\left\Vert \PLNE\right\Vert\supseteq \mathbb{CU}$) in two distinct ways---it is instructive to see both of these proofs as they break down the characteristics of union-closed convex properties in different ways and feature distinct (if similar) characteristic formulas. As in Section \ref{c:section:convex_prop}, we assume all properties are over a fixed finite set of propositional letters $\mathsf{X}\subseteq\mathsf{P}$. 

For the first proof, we use the formula $\Bot=\bot \land \NE$ for the empty property, and, to express nonempty convex properties, construct a $U$-schema (recall Section \ref{c:section:convex_prop}) as well as an $F$-schema in $\PLNE$, where $\chi^\mathrm{F}_\QQQ$ is an \emph{$F$-schema} if for any nonempty $\PPP$ and $t\subseteq 2^\mathsf{X}$,
\begin{align*}
    t\vDash \chi^\mathrm{F}_\PPP  \iff t\subseteq \bigcup \PPP.
\end{align*}
Recall that a $U$-schema yields characteristic formulas for nonempty upward-closed properties. An $F$-schema, on the other hand, yields characteristic formulas for flat properties: for flat $\PPP$, $\left\Vert \chi^{\mathrm{F}}_\PPP\right\Vert=\PPP$. The conjunction of the two gives characteristic formulas for nonempty union-closed convex properties:

\begin{lemma} \label{c:lemma:convex_uc_char_formula}
Let $\chi^{\mathrm{F}}_\QQQ$ be an $F$-schema and $\chi^{\mathrm{U}}_\QQQ$ be a $U$-schema. Then for any nonempty union-closed convex $\PPP$, $\left\Vert \chi^{\mathrm{F}}_\PPP\land \chi^{\mathrm{U}}_\PPP\right \Vert=\PPP  $.
\end{lemma}
\begin{proof}
$\supseteq$: For any $t\in \PPP$, $t\subseteq t\subseteq \bigcup\PPP$, whence $t\vDash \chi^{\mathrm{F}}_\PPP\land \chi^{\mathrm{U}}_\PPP$.

$\subseteq$: If $t\vDash \chi^{\mathrm{F}}_\PPP\land \chi^{\mathrm{U}}_\PPP$, then for some $s\in \PPP$, $s\subseteq t\subseteq 
\bigcup \PPP$. We have $\bigcup \PPP\in \PPP$ by union closure and the fact that $\PPP\neq \emptyset$, whence $t\in \PPP$ by convexity.
\end{proof}

Clearly, given $\filleddiamond \phi\equiv (\phi\land \NE)\vee \top$, we can construct a $U$-schema $\chi^{\mathrm{U}}_\QQQ$ in $\PLNE$ in essentially the same manner as in Lemma \ref{c:lemma:upward_char_formula}. For the $F$-schema $\chi^{\mathrm{F}}_\QQQ$, we have the following (which can also be used to prove that $\PLV$ is expressively complete for the class of all flat properties \cite{yang2017}):

\begin{lemma}\label{c:lemma:flat_char_formulas}
For nonempty $ \PPP$, let $$\chi^{\mathrm{F}}_\PPP:=\bigvee_{s\in \PPP}\chi_s.$$
Then for any $t\subseteq 2^{\mathsf{X}}$, $t\vDash \chi^{\mathrm{F}}_\PPP  \iff t\subseteq \bigcup \PPP$.
\end{lemma}
\begin{proof}
    Letting $\PPP=\{s_1,\ldots, s_n\}$, we have that 
    \begin{align*}
        t\vDash\bigvee_{s\in \PPP}\chi_s &\iff [\exists t_1,\ldots, t_n :t=\bigcup_{i=1}^n t_i\text{ and }t_i\vDash \chi_{s_i}\text{ for each }1 \leq i\leq n]\\
        &\iff [\exists t_1,\ldots, t_n :t=\bigcup_{i=1}^n t_i\text{ and }t_i\subseteq s_i \text{ for each }1 \leq i\leq n]\iff t\subseteq \bigcup \PPP. \tag*{\qedhere}
    \end{align*}
\end{proof}

We have finished proving the required lemmas for the first proof. 

For the second proof, we make the following case distinction, now also as to whether or not $\PPP$ has the empty team property.

\begin{lemma}\label{c:lemma:convex_union_closed_characteristic_formulas} For each $\PPP=\{t_1,\ldots, t_n\}\in \mathbb{CU}$:
\begin{enumerate}
    \item[(a)] \label{c:lemma:convex_union_closed_characteristic_formulas_empty} If $\PPP$ is the empty property, $\PPP=\left\Vert \Bot\right\Vert$.
    \item[(b)] \label{c:lemma:convex_union_closed_characteristic_formulas_empty_team_property} If $\PPP$ has the empty team property, $\PPP=\left\Vert \bigvee_{s\in \PPP} \chi_s\right\Vert$.
    \item[(c)] \label{c:lemma:convex_union_closed_characteristic_formulas_no_empty_team_property} If $\PPP\neq \emptyset$ and $\PPP$ does not have the empty team property,
    $$\PPP=\left\Vert\bigvee_{v_1\in t_1,\ldots,v_n\in t_n}((\chi_{v_1}\vee \cdots \vee \chi_{v_n})\land \NE)\right\Vert.$$
\end{enumerate}
\end{lemma}
\begin{proof}
    Item (a) is obvious. For item (b), note that then (cf. Fact \ref{c:fact:closure_props}) $\PPP$ is downward closed and hence also flat. By Lemma \ref{c:lemma:flat_char_formulas}, $t\vDash \bigvee_{s\in \PPP} \chi$ iff $t\subseteq \bigcup \PPP$, and, if $\PPP$ is flat, clearly $t\subseteq \bigcup \PPP$ iff $t\in \PPP$.

    We now show item (c). For the direction $\subseteq$, let $t_i\in \PPP$. For each $v_i\in t_i$, we have that $\{v_i\}\vDash \chi_{v_i}\land \NE$, so that also (using the empty team property of classical formulas), $\{v_i\}\vDash (\chi_{v_1}\vee \cdots \vee \chi_{v_n})\land \NE$ for any choice of $v_1\in t_1,\ldots, v_{i-1}\in t_{i-1},v_{i+1}\in t_{i+1},\ldots, v_{n}\in t_{n}$. Therefore,
    
    $$\{v_i\}\vDash \bigvee_{v_1\in t_1, \ldots, v_{i-1}\in t_{i-1}, v_{i+1}\in t_{i+1},\ldots, v_{n}\in t_{n}}((\chi_{v_1}\vee \cdots \vee \chi_{v_n})\land \NE),$$
whence $t_i=\bigcup_{v_i\in t_i} \{v_i\}\vDash\bigvee_{v_1\in t_1,\ldots,v_n\in t_n}((\chi_{v_1}\vee \cdots \vee \chi_{v_n})\land \NE)$.

For the direction $\supseteq$, let $s\vDash \bigvee_{v_1\in t_1,\ldots,v_n\in t_n}((\chi_{v_1}\vee \cdots \vee \chi_{v_n})\land \NE)$. By the fact that $\PPP$ does not have the empty team property, no $t_i\in \PPP$ is the empty team. We then have that $s=\bigcup_{v_1\in t_1,\ldots,v_n\in t_n} t_{v_1,\ldots,v_n}$ where $t_{v_1,\ldots,v_n}\vDash (\chi_{v_1}\vee \cdots \vee \chi_{v_n})\land \NE$. (Note that if $\PPP$ did have the empty team property, we would have $\bigvee_{v_1\in t_1,\ldots,v_n\in t_n}((\chi_{v_1}\vee \cdots \vee \chi_{v_n})\land \NE)\equiv \bigvee \emptyset \equiv\bot$.)

We show that $s\subseteq \bigcup \PPP$ and that $t_i\subseteq s$ for some $t_i\in \PPP$. Since $\bigcup \PPP\in \PPP$ by the union closure of $\PPP$ (as well as the fact that $\PPP\neq \emptyset$), we will then have $t_i\in \PPP$ by the convexity of $\PPP$.

$s\subseteq \bigcup \PPP$: We have that $s=\bigcup_{v_1\in t_1,\ldots,v_n\in t_n} t_{v_1,\ldots,v_n}$, and for each $t_{v_1,\ldots,v_n}$, we have, by $t_{v_1,\ldots,v_n}\vDash (\chi_{v_1}\vee \cdots \vee \chi_{v_n})\land \NE$, that $t_{v_1,\ldots,v_n}\subseteq \bigcup_{1\leq i \leq n}\{v_i\}\subseteq \bigcup \PPP$.

$t_i\subseteq s$ for some $t_i\in \PPP$: Assume for contradiction that $t_i\not\subseteq s$ for all $t_i \in \PPP$. Then for each $t_i\in \PPP$ there is some $w_i\in t_i$ such that $w_i\notin s$. We have $t_{w_1,\ldots,w_n}\vDash (\chi_{w_1}\vee \cdots \vee \chi_{w_n})\land \NE$ so $t_{w_1,\ldots,w_n}\subseteq \bigcup_{1\leq i \leq n}\{w_i\}$ and $t_{w_1,\ldots,w_n}\neq \emptyset$. We also have that $t_{w_1,\ldots,w_n}\subseteq s$, whence $s\cap \bigcup_{1\leq i \leq n}\{w_i\}\neq \emptyset$, contradicting the fact that $w_i\notin s$ for all $t_i\in \PPP$.
\end{proof}

Putting together Proposition \ref{c:prop:PLNE_convex_union_closed}, the Lemmas \ref{c:lemma:convex_uc_char_formula}, \ref{c:lemma:upward_char_formula}, and \ref{c:lemma:flat_char_formulas} (for the first proof), and Lemma \ref{c:lemma:convex_union_closed_characteristic_formulas} (for the second proof), we have shown:

\begin{theorem}\label{c:theorem:NE_expressive_completeness}
    $\PLNE$ is expressively complete for convex union-closed properties: $$\left\Vert \PLNE\right\Vert=\mathbb{CU}.$$
\end{theorem}
Notably, an analogous result was obtained independently in the setting of context-guarded team logics by Marius Tritschler in his Master's thesis \cite{tritschler}, where he introduces a closure property he calls `partial union property', which, in our setting, would translate to:
    $$\text{If $t\vDash \varphi$ and $t'\vDash \varphi$, then $t\cup s'\vDash \varphi$ for all $s'\subseteq t'$.}$$
It is readily seen that a formula $\varphi$ satisfies the partial union property if and only if it is convex and [finite] union-closed.

To conclude this section, let us comment on the relationship between the connectives we have studied that break downward closure: $\NE$ and $\filleddiamond$. As we have observed, in a setting with $\NE$ and $\vee$ (and $\top$ and $\land$), $\filleddiamond \phi$ is definable as $(\phi\land \NE)\vee \top$. On the other hand, $\NE$ is definable using $\filleddiamond$ (and $\top$) as $\filleddiamond \top$. Therefore, it follows that the extension of $\PLV$ with $\filleddiamond$ is also expressively complete for $\mathbb{CU}$.

However, swapping out $\filleddiamond$ with $\NE$ in any of our convex logics results in a logic that, while convex, is no longer complete for $\mathbb{C}$---an easy induction shows that these logics are \emph{downward closed modulo the empty team}: for each formula $\phi$ of one of the logics, if $t\vDash \phi$, $s\subseteq t$, and $s\neq \emptyset $, then $s\vDash \phi$. And clearly there are convex properties that are not downward closed modulo the empty team (e.g., $\left \Vert \filleddiamond p \land \filleddiamond \lnot p\right \Vert \in \mathbb{C}$). (Note that with the convex-closure disjunction $\vee_C$ mentioned in endnote \hyperref[endnote:convex_closure_disjunctions]{3}, we have $\filleddiamond \phi \equiv (\phi \land \NE)\vee_C \top$, whereas the result of swapping out $\vveedot$ with $\vvee_C$ and $\filleddiamond$ with $\NE$ in $\CONINQ$ is downward closed modulo the empty team.)

We next study convex and union-closed properties in the modal environment.

\subsection{Modal Properties}
\label{c:section:union-closed_convex_modal}

We define the syntax and the semantics of the modal extensions $\MLLVNE$/$\MLGVNE$ in the obvious way. In contrast with the extensions featuring the modalities $\diamonddiamond$ and $\boxbox$ defined in Section \ref{c:section:convex_modal}, the extension $\MLGVNE$ \emph{is} convex. The situation is similar to that with $\vee$: while $\diamonddiamond$ does not preserve convexity  across the board, it does preserve it in a convex and union-closed setting:

\begin{proposition}\label{c:prop:MLNE_convex_union_closed}
    Formulas of $\MLLVNE$ and $\MLGVNE$ are convex and union closed.
\end{proposition}
\begin{proof}
    By induction on the structure of formulas $\phi$. Most cases are straightforward; note in particular that $\Diamond \phi$ and $\Box\phi$ are flat and therefore convex and union closed. We only show explicitly that (i) $\diamonddiamond \phi$ is union closed provided that $\phi$ is union closed; and (ii) $\diamonddiamond \phi$ is convex provided that $\phi$ is convex and union closed.
    
    For (i), let $T\neq \emptyset$ be such that for all $t\in T$, $M,t\vDash \diamonddiamond\phi$. Then for each $t\in T$, there is an $s_t\subseteq W$ such that $tRs_t$ and $M,s_t\vDash \phi$. By union closure, $M,\bigcup_{t\in T}s_t\vDash \phi$. By $s_t\subseteq R[t]$ for all $t\in T$, it follows that $\bigcup_{t\in T}s_t\subseteq \bigcup_{t\in T} R[t]=R[\bigcup T]$, and by $t\subseteq R^{-1}[s_t]$ for all $t\in T$, it follows that $\bigcup T=\bigcup_{t\in T} t\subseteq \bigcup_{t\in T}R^{-1}[s_t]=R^{-1}[\bigcup_{t\in T}s_t] $; therefore $\bigcup T R \bigcup_{t\in T}s_t$ whence $M,\bigcup T\vDash \diamonddiamond\phi$.

    For (ii), let $M,t\vDash \diamonddiamond\phi$ and $M,s\vDash \diamonddiamond\phi$ and $s\subseteq u \subseteq t$. Then there are $t',s'\subseteq W$ such that $tRt'$, $sRs'$, $M,t'\vDash \phi$, and $M,s'\vDash \phi$. By union closure, $M,t'\cup s'\vDash \phi$. Define $u':=(t'\cup s')\cap R[u]$. We will show (a) $uRu'$ and (b) $s'\subseteq u'$. Then by (b) we will have $s'\subseteq u'\subseteq t'\cup s'$, so by convexity, $M,u'\vDash \phi$; so that by (a) we will have $M,u\vDash \diamonddiamond\phi$.

    For (a), clearly $u'\subseteq R[u]$. To show that $u\subseteq R^{-1}[u']$, suppose that $x\in u$. Then since $u\subseteq tRt'$, we have $x\in u\subseteq t\subseteq R^{-1}[t']\subseteq R^{-1}[t'\cup s']$. So there is $y\in t'\cup s'$ with $xRy$, whence also $y\in R[u]$, and therefore $y$ witnesses $x\in R^{-1}[u']$.    

    For (b), since $u\supseteq sRs'$, we have $s'\subseteq R[s]\subseteq R[u]$, hence $s'\subseteq (t'\cup s')\cap R[u]=u'$.
\end{proof}

We can further  prove a modal version of Theorem \ref{c:theorem:NE_expressive_completeness} for each of these extensions. As in Section \ref{c:section:convex_modal}, we omit the details; the proofs are completely analogous to the propositional  case. 

\begin{theorem}\label{c:theorem:MLNE_expressive_completeness}
     Both $\MLLVNE$ and $\MLGVNE$  are expressively complete for the class of all convex union-closed modal properties invariant under bounded bisimulation.
\end{theorem}

Aloni's Bilateral State-Based Modal ($\BSML$) \cite{aloni2022} is essentially $\MLLVNE$ augmented with a \emph{bilateral negation} which does not affect the expressive power of the logic (with respect to properties expressed by formulas; there is another sense in which the expressive power \emph{is} affected---see
\cite{aloni2023,anttila2024}). Therefore, the above also establishes that $\BSML$ is  expressively complete for the class of all convex union-closed modal properties invariant under bounded bisimulation; this solves a problem that was left open in \cite{aloni2023}.
\section{Definability, Extensions and Translations}
\label{c:section:uniform_definability}

In this final part of the paper, we consider a different set of questions altogether. We no longer ask questions of expressive completeness, but, motivated by the above findings, we propose and compare a number of definability and extension criteria. 

Specifically, we generalize the notion of uniform definability~\cite{kontinen2009,ciardelli2009,galliani2013b,yang2014,yang20173,ciardelli2019,hella2024,barbero2024} to articulate and distinguish multiple ways in which one team logic may extend another. The most significant extension notion we introduce is that of \emph{compositional extension}---we show that the convex logics $\CONDEP$ and $\CONINQ$ are compositional extensions of the downward-closed logics $\DEP$ and $\INQ$, respectively, and that a logic is a compositional extension of another logic precisely when there is a compositional translation from the latter to the former.

\sloppy
\subsection{Motivation} Recall that dependence logic $\DEP $ and inquisitive logic $\INQ$ are expressively equivalent: both are complete for nonempty downward-closed properties, i.e. $$\left\Vert\DEP\right\Vert=\left\Vert\INQ\right\Vert=\mathbb{DE}.$$ 
\fussy
Hence, as the split disjunction $\vee$ preserves downward closure and the empty team property, it holds that for all $\phi,\psi$ of $\DEP$ or $\INQ$:
    $$\left\Vert \phi\vee \psi\right\Vert\in \left\Vert\INQ\right\Vert.$$
And more generally, each property expressed by a split disjunction of downward-closed formulas is expressible in $\INQ$. But even if for each split-disjunction \emph{instance} $\phi\vee \psi$ of $\INQ$-formulas (or $\DEP$-formulas) there is some $\INQ$-formula $\theta_{\phi\vee\psi}$ equivalent to $\phi\vee\psi$, it can be shown \cite{ciardelli2019} that $\vee$ is not \emph{uniformly definable} in $\INQ$: there is no \emph{context} $\theta_\vee[\cdot_1,\cdot_2]$ of $\INQ$ such that for any $\phi,\psi$ of $\INQ$, $\phi\vee\psi\equiv \theta_\vee[\phi,\psi]$. (Here a context is a formula with designated atoms $\cdot_i$; $\theta_\vee[\cdot_1,\cdot_2]$ denotes that the designated atoms in $\theta_\vee$ are among $\cdot_1,\cdot_2$; and $\theta_\vee[\psi_1,\psi_2]:=\theta_\vee(\psi_1/\cdot_1,\psi_2/\cdot_2)$). To make matters precise:
\begin{definition}[Uniform definability] \label{c:definition:uniform_definability}
        An $n$-ary logical connective $\heart$ is \emph{uniformly definable} in a logic $\LOGIC$ if there exists a context $\theta_{\heart}[\cdot_1,\ldots,\cdot_n]$ of $\LOGIC$ such that for all $\phi_1,\ldots, \phi_n$ in the language of $\LOGIC$, $\heart[\phi_1,\ldots,\phi_n]\equiv \theta_{\heart}[\phi_1,\ldots,\phi_n]$.
\end{definition}
Likewise has it been shown \cite{yang20173,barbero2024} that neither $\vvee$ nor $\to$ is uniformly definable in $\DEP$ (while, again, each instance $\left\Vert\phi\vvee\psi\right\Vert$ and each instance $\left\Vert\phi\to\psi\right\Vert$, where $\phi,\psi\in \DEP$, is expressible in $\DEP$).

\begin{remark}[Excursion: expressibility, definability, and substitution]
    This disconnect between expressibility and uniform definability relates to the failure of \emph{closure under uniform substitution} in team logics~\cite{yang20173,ciardelli2019,ciardellibook}, i.e. $\phi\vDash \psi$ need not imply $\phi(\chi/p)\vDash \psi(\chi/p)$. (For instance, $p\vee p\vDash p$, whereas $(p\vvee \lnot p)\vee (p\vvee \lnot p)\nvDash (p\vvee \lnot p)$ (consider the team $\{v_{p},v_{\overline{p}}\}$), and also $p\land (q \vee r)\vDash (p\land q)\vee (p\land r)$, whereas \\ $\filleddiamond \top \land (q\lor r)\nvDash (\filleddiamond \top\land q)\vee (\filleddiamond \top\land r)$ (consider the team $\{v_{q\overline{r}}\}$).) 
    
    To indicate why, we present an adaptation of the explanations in \cite{ciardelli2019,ciardellibook}. For $\LOGIC_1$ and $\LOGIC_2$ team logics, if each connective of $\LOGIC_1$ is uniformly definable in $\LOGIC_2$ and vice versa, each property expressible in one is also expressible in the other, so $\left\Vert \LOGIC_1\right\Vert = \left\Vert \LOGIC_2\right\Vert$. Conversely, \emph{given} closure under uniform substitution, $\left\Vert \LOGIC_1\right\Vert = \left\Vert \LOGIC_2\right\Vert$ implies the uniform definability of each connective of one logic in the other. For if $\LOGIC_1$ can express, say, $\left\Vert p \heart q \right\Vert$ where $\heart$ is a binary connective of $\LOGIC_2$, then there is a context $\theta_{\heart}[\cdot_1,\cdot_2]$ of $\LOGIC_1$ such that $ \theta_{\heart}[p,q]\equiv p\heart q$ (we are assuming here that $\LOGIC_1$ and $\LOGIC_2$ have the locality property; see Definition~\ref{def:locality}). Then for any formulas $\phi,\psi$ of $\LOGIC_1$, by closure under uniform substitution, also $ \theta_{\heart}[\phi,\psi]\equiv \phi\heart \psi$.  
\end{remark}

To connect this discussion with our concerns, let us first recall Fact \ref{c:fact:disjunctions_convexity_break}---the following records (essentially) the same fact, reformulated and reproved in a more abstract manner. Given properties $\mathcal{P}$ and $\mathcal{Q}$, we write $\mathcal{P}\vee\mathcal{Q}:=\{t\cup s\mid t\in \PPP$ and $s\in \mathcal{Q}\}$ and $\mathcal{P}\vvee\mathcal{Q}:=\PPP\cup \QQQ$; and given classes of properties $\mathbb{P}$ and $\mathbb{Q}$, we write $\mathbb{P}\vee\mathbb{P}\subseteq \mathbb{Q}$ if $\PPP,\PPP'\in \mathbb{P}$ implies $\PPP\vee\PPP'\in \mathbb{Q}$, and similarly for $\mathbb{P}\vvee\mathbb{P}\subseteq \mathbb{Q}$.
\begin{fact} \label{c:fact:disjunctions_convexity_break2}
    (i) $\mathbb{C}\vee\mathbb{C}\not\subseteq \mathbb{C}$, and (ii) $\mathbb{C}\vvee\mathbb{C}\not\subseteq \mathbb{C}$.
\end{fact}
\begin{proof}
 (i) Let $\mathcal{P}\mathrel{:=}\{\{v_1\},\{v_2,v_3\}\}$ and $\mathcal{Q}\mathrel{:=}\{\{v_1\}\}$ (where the $v_i$ are valuations). Then $\mathcal{P}, \mathcal{Q}\in \mathbb{C}$, but $\mathcal{P}\lor \mathcal{Q}= \{\{v_1\},\{v_1, v_2, v_3\}\}\notin \mathbb{C}$.

    (ii) Let $\mathcal{P}\mathrel{:=}\{\{v_1,v_2,v_3\}\}$ and $\mathcal{Q}\mathrel{:=}\{\{v_1\}\}$. Then $\mathcal{P}, \mathcal{Q}\in \mathbb{C}$, but we have that $\mathcal{P}\vvee\mathcal{Q}= \{\{v_1\},\{v_1, v_2, v_3\}\}\notin \mathbb{C}$.
 \end{proof}
 
It follows that if a logic $\LOGIC$ is expressively complete for all convex properties (i.e., $\left\Vert\LOGIC\right\Vert= \mathbb{C}$), then there are formulas $\varphi, \psi$ in the language of $\LOGIC$ such that $\varphi\lor\psi$ is not expressible (i.e., $\left\Vert\varphi\lor\psi\right\Vert\notin\left\Vert\LOGIC\right\Vert$). In particular, as $\left\Vert\CONDEP\right\Vert=\mathbb{C}$ (cf. Theorem \ref{c:theorem:convex_expressive_completeness}), there are formulas $\phi,\psi$ of $\CONDEP$ such that $\left\Vert \phi\vee\psi\right\Vert \notin \left\Vert \CONDEP \right\Vert$; \emph{a fortiori}, the split disjunction is not uniformly definable in $\CONDEP$.

In contrast, we have that $\mathbb{D}\vee\mathbb{D}\subseteq \left\Vert\CONDEP \right\Vert$. In fact, given that for downward-closed $\phi$ and $\psi$, $\phi\vee\psi\equiv \phi\veedot \psi$, the split disjunction of any two downward-closed formulas \emph{is} definable in $\CONDEP$ by the context $\theta_\vee[\cdot_1,\cdot_2]:=\cdot_1\veedot \cdot_2$. Specifically, when restricted to (formulas defining properties in) $\left\Vert\DEP\right\Vert=\mathbb{DE}$, the split disjunction is definable in $\CONDEP$, hence each connective of dependence logic $\DEP$ is definable in $\CONDEP$ when restricted to $\left\Vert\DEP\right\Vert$.

And so, while $\CONDEP$ cannot define the connectives of $\DEP$ \textit{within the setting of $\CONDEP$}, it can within that of $\DEP$.

As considered in the literature, the concept of uniform definability in a logic $\LOGIC$ applies only to definability over formulas of $\LOGIC$. In order to formalize the observations above, we will therefore  generalize this notion to allow for definability in one logic over formulas of a different logic, as well as for definability over semantic objects rather than over formulas of a specific logic. This is the aim of the rest of this section.

 The novel definability notions allow us to distinguish multiple ways in which one logic may be said to extend another. We call the manner in which $\CONDEP$ extends $\DEP$ \emph{compositional extension}. This choice of terminology is motivated by its connection to \emph{compositional translations}: we will show that a logic $\LOGIC_2$ is a compositional extension of a logic $\LOGIC_1$ just in case there is a compositional translation from $\LOGIC_1$ to $\LOGIC_2$, in the following sense (cf. the similar notions and discussion in  \cite{hodges2016remakrs,yang20173,rosetta,peterswesterstahl}):

\begin{definition}[Compositional translation] \label{c:def:compositional_translation}
      A map $\tau:\LOGIC_1\to \LOGIC_2$ is a \emph{compositional translation} or \emph{translational embedding} if it satisfies the following two conditions. 
    \begin{enumerate}
        \item $\tau$ is \emph{meaning-preserving}: for all formulas $\varphi$ of $\LOGIC_1$, $ \varphi \equiv \tau(\varphi) $. 
        \item $\tau$ is \emph{compositional}: for each connective $\heart$ of $\LOGIC_1$, there is a context $\theta_{\heart}$ of $\LOGIC_2$ such that $\tau(\heart[\varphi_1, \ldots, \varphi_n])=\theta_{\heart}[\tau(\varphi_1), \ldots, \tau(\varphi_n)]$ for all $\varphi_1, \ldots, \varphi_n$ of $\LOGIC_1$.\footnote{Note that if $\tau:\LOGIC_1\to \LOGIC_2$ is a translational embedding, then $\tau$ is, in particular, a translation in the sense of being \emph{full} and \emph{faithful}: for all formulas $\varphi, \psi$ of $\LOGIC_1$, we have
    \[
        \varphi\vDash_{\LOGIC_1}\psi \quad \text{iff}\quad \tau(\varphi)\vDash_{\LOGIC_2} \tau(\psi). 
    \]
    This follows from meaning-preservation alone. Compositionality, in turn, licenses the homomorphic character of the term `embedding'.}
    \end{enumerate}
\end{definition}
\noindent The other extension notions we introduce similarly correspond to translation notions.

\subsection{Extensions and Translations, briefly}
We first present an abstract, simplified, and entirely semantic version of the content of the following sections. This is to highlight the guiding ideas and provide a roadmap without details getting in the way.\\

\paragraph{Abstract logics.} Throughout Section \ref{c:section:uniform_definability}, we write $\mathbb{A}$ for the class of all properties over $\mathsf{P}$,  simply $\llbracket\phi \rrbracket$ for $\left\Vert\phi \right\Vert_\mathsf{P}$, and $\mathbb{D}$ for the class of downward-closed properties in $\mathbb{A}$ (similarly for $\mathbb{C}$, etc.). An \emph{abstract propositional team logic} $\LOGIC$ is a pair $(\mathrm{A}(\LOGIC),\mathrm{C}(\LOGIC))$, where $\mathrm{A}(\LOGIC)\subseteq \mathbb{A}$ is a set representing (the denotations of) the \emph{primitives} or \emph{atomic formulas} of $\LOGIC$ (e.g., $\llbracket p\rrbracket$); and $\mathrm{C}(\LOGIC)$ is a set of functions $\heart:\mathbb{A}^n\to \mathbb{A}$ ($n \in \mathbb{N}$) representing (the denotations of) connectives (e.g., $\land:\mathbb{A}^2\to \mathbb{A}, (\mathcal{P},\mathcal{P}')\mapsto \mathcal{P}\cap \mathcal{P}'$). We think of logical constants and formulas/properties as nullary functions into $\mathbb{A}$ (e.g., $\llbracket \bot\rrbracket=\{\varnothing\}\in \mathbb{A}$). The set of \emph{contexts} $\mathrm{Cxt}(\LOGIC)$ of an abstract logic $\LOGIC$ is the closure of $\mathrm{A}(\LOGIC)\cup \mathrm{C}(\LOGIC)$ under  composition (e.g., $\land(\llbracket p\rrbracket,\cdot):\mathbb{A}\to \mathbb{A}, \mathcal{P}\mapsto \llbracket p \rrbracket\cap \mathcal{P}$). We denote the set of nullary contexts of $\LOGIC$ by $\llbracket\LOGIC\rrbracket$; these represent the formulas of $\LOGIC$.\\

\paragraph{Generalized (uniform) definability.} Given any $n$-ary context $\heart:\mathbb{A}^n\to \mathbb{A}$ and $\mathbb{P}\subseteq \mathbb{A}$, we write $\heart^\mathbb{P}$ for the restricted function $\heart\restriction \mathbb{P}^n$; and for sets such as $\mathrm{C}(\LOGIC)$, we write $\mathrm{C}^\mathbb{P}(\LOGIC)\mathrel{:=}\{{\heart}^\mathbb{P}\mid {\heart}\in \mathrm{C}(\LOGIC)\}$. Note that for nullary $\heart$, $\heart^\mathbb{P}$ is always $\heart$. We say that a function $\heart$ is \emph{(uniformly) defined with respect to $\mathbb{P}$} by a function $\medstar$ if $\medstar^\mathbb{P}=\heart^\mathbb{P}$. We say that a function $\heart$ is \emph{(uniformly) definable with respect to $\mathbb{P}$} in $\LOGIC$ if there is a context $\medstar$ of $\LOGIC$ that defines $\heart$ with respect to $\mathbb{P}$. Note that uniform definability in $\LOGIC$ in the sense of Definition \ref{c:definition:uniform_definability} corresponds to definability in $\LOGIC$ with respect to $\llbracket\LOGIC\rrbracket$.\\

\paragraph{Compositional extensions.} For abstract logics $\LOGIC_1$ and $\LOGIC_2$, $\LOGIC_2$ is a \emph{compositional extension} of $\LOGIC_1$ if each context $\heart$ of $\LOGIC_1$ is definable in $\LOGIC_2$ with respect to $\llbracket\LOGIC_1\rrbracket$  (equivalently, each element of $\mathrm{A}(\LOGIC_1)\cup \mathrm{C}(\LOGIC_1)$ is definable in $\LOGIC_2$ with respect to $\llbracket\LOGIC_1\rrbracket$; or equivalently and informally, $\LOGIC_2$ can express both the atomic denotations of $\LOGIC_1$, i.e. $\mathrm{A}(\LOGIC_1)\subseteq \mathrm{Cxt}(\LOGIC_2)$, and the connective denotations of $\LOGIC_1$ as understood by $\LOGIC_1$, i.e. $\mathrm{C}^{\llbracket\LOGIC_1\rrbracket}(\LOGIC_1)\subseteq \mathrm{Cxt}^{\llbracket\LOGIC_1\rrbracket}(\LOGIC_2)$).\\

\paragraph{Compositional translations.} Within this setting, a map  $\tau:\llbracket\LOGIC_1\rrbracket\to \llbracket\LOGIC_2\rrbracket$ is then a \emph{compositional translation} if 
\begin{enumerate}
    \item For all $\PPP\in \llbracket\LOGIC_1\rrbracket$, $\tau(\PPP)=\PPP$. 
    \item For each connective $\heart\in \mathrm{C}(\mathrm{\LOGIC_1})$, there is a context $\theta_{\heart}$ of $\LOGIC_2$ of the same arity such that for all $\PPP_1,\ldots, \PPP_n\in \llbracket\LOGIC_1\rrbracket$, $\tau(\heart(\PPP_1,\ldots,\PPP_n))=\theta_{\heart}(\tau(\PPP_1),\ldots,\tau(\PPP_n))$.
\end{enumerate}

We show that $\LOGIC_2$ is a compositional extension of $\LOGIC_1$ if and only if there is a compositional translation $\tau:\llbracket\LOGIC_1\rrbracket\to \llbracket\LOGIC_2\rrbracket$.

\begin{proof}
Suppose that $\LOGIC_2$ is a compositional extension of $\LOGIC_1$. For $\PPP\in \llbracket\LOGIC_1\rrbracket$, let
    \begin{align*}
        \tau(\PPP) \mathrel{:=} 
        \begin{cases}
            \PPP, & \text{if } \PPP\in \mathrm{A}(\LOGIC_1), \\
            \theta_{\heart}(\tau(\PPP_1), \ldots, \tau(\PPP_n)) & \text{if $\PPP\notin \mathrm{A}(\LOGIC_1)$ so } \PPP=\heart(\PPP_1, \ldots, \PPP_n),
        \end{cases}
    \end{align*}
    where for the $\heart\in \mathrm{C}(\LOGIC_1)$, $\theta_{\heart}$ is chosen such that ${\heart}^{\llbracket \LOGIC_1\rrbracket}=\theta_{\heart}^{\llbracket \LOGIC_1\rrbracket}$, which is possible since $\mathrm{C}^{\llbracket\LOGIC_1\rrbracket}(\LOGIC_1)\subseteq \mathrm{Cxt}^{\llbracket\LOGIC_1\rrbracket}(\LOGIC_2)$. Observe that $\tau:\llbracket\LOGIC_1\rrbracket\to \llbracket\LOGIC_2\rrbracket$ is a well-defined function, as can be readily seen by an induction that simultaneously establishes that $\tau$ meets condition 1 (i.e. it is the inclusion $\tau:\PPP\mapsto \PPP$). One uses $\mathrm{A}(\LOGIC_1)\subseteq \mathrm{Cxt}(\LOGIC_2)$ (so $\mathrm{A}(\LOGIC_1)\subseteq \llbracket\LOGIC_2\rrbracket$) for the base case and ${\heart}^{\llbracket \LOGIC_1\rrbracket}=\theta_{\heart}^{\llbracket \LOGIC_1\rrbracket}$ for the inductive case. Condition 2 then follows easily.

    Conversely, suppose $\tau:\llbracket\LOGIC_1\rrbracket\to \llbracket\LOGIC_2\rrbracket$ is a compositional translation. Then we have $\mathrm{A}(\LOGIC_1)\subseteq \mathrm{Cxt}(\LOGIC_2)$ because
    \[
        \mathrm{A}(\LOGIC_1)\subseteq \llbracket\LOGIC_1\rrbracket=\tau(\llbracket\LOGIC_1\rrbracket)\subseteq \llbracket\LOGIC_2\rrbracket\subseteq     \mathrm{Cxt}(\LOGIC_2).
    \]
    Next for $\heart\in \mathrm{C}(\LOGIC_1)$ of arity $n$, take $\theta_{\heart}$ given by condition 2 for $\tau$. We claim that
    \[
    \heart^{\llbracket \LOGIC_1 \rrbracket}=\theta_{\heart}^{\llbracket \LOGIC_1 \rrbracket}.
    \]
    To see this, let $\PPP_1, \ldots, \PPP_n\in \llbracket\LOGIC_1\rrbracket$ be arbitrary. By condition 1, we have $\PPP_i=\tau(\PPP_i)$ for $1\leq i \leq n$. Thus it suffices to verify that $\heart(\PPP_1, \ldots, \PPP_n)=  \theta_{\heart}(\tau(\PPP_1), \ldots, \tau(\PPP_n))$, and this immediately follows by conditions 1 and 2.
\end{proof}

\subsection{Set-Up}\label{section:set-up} While the above sets out and connects the key new concepts, it also makes a number of simplifying assumptions. For example, it does not inform us how to handle cases in which either $\heart$ or $\theta_{\heart}$ is undefined for some formulas of a logic (e.g., $\con{\con{p}}$ is undefined in $\DEP$).\footnote{\label{endnote:dependence}In the papers in which this question is most relevant (\cite{yang20173,barbero2024})---relevant because they deal with contexts in  dependence logic $\DEP$, which features what we would classify as restricted-domain connectives ($\lnot$ and $=_n\hspace{-0.1cm}(\cdot)$)---this issue is sidestepped by only considering definability of connectives $\heart$ such that $\heart[\phi_1,\ldots,\phi_n]$ is always defined (with contexts defined in a corresponding way: designated atoms may only appear in positions in which they may be replaced by any formula). The negation $\lnot$ and the connectives $=_n\hspace{-0.1cm}(\cdot)$, then, are not viewed as connectives---instead, all formulas $\lnot p$ and $\dep{p_1,\ldots }{p_n}$ are classified as primitives. However, with this restriction, the definition does not apply to certain definability facts which are arguably of the same sort as those the definition does encompass. For instance, the connectives $=_n\hspace{-0.1cm}(\cdot)$ are, in a clear sense that is analogous to Definition \ref{c:definition:uniform_definability}, definable in $\INQ$ given the following (first observed essentially in \cite{abramsky}):
    $$\dep{p_1,\ldots}{p_n}\equiv ((p_1\vvee\lnot p_1)\land \cdots \land (p_{n-1}\vvee \lnot p_{n-1}))\to (p_n\vvee \lnot p_n) .$$
Indeed, \cite{yang20173}, despite officially employing Definition \ref{c:definition:uniform_definability}, also essentially applies the notion in this wider sense when it notes this definability fact. Our generalization will allow for connectives and contexts with restricted domains, and will cover definability results of this kind.}

To account for this and more, we present a more detailed way of specifying the syntax of a propositional team logic (cf. the similar accounts in \cite{yang20173,hodges2001compositionality}). We then make explicit some assumptions on the semantics, and only after that (in Sections~\ref{subsec:definability}--\ref{subsec:compositionaltra}) do we return to discuss definability, extensions and translations.

\begin{definition}[Propositional team logics] \label{c:def:abstract_logic} A propositional team logic $\LOGIC$ is a pair $(\mathrm{A}(\LOGIC),\mathrm{C}(\LOGIC))$ where:
\begin{itemize}
    \item[-] $\mathrm{A}(\LOGIC)$ is a set of \emph{primitives} or \emph{atomic formulas};
    \item[-] $\mathrm{C}(\LOGIC)$ is a set of \emph{connectives}, each with an arity $n\in \mathbb{N}$ and a formation rule specifying, for all $\phi_1,\ldots,\phi_n$, whether $\heart[\phi_1,\ldots,\phi_n]$ is well-defined.
\end{itemize}
    The set of formulas of $\LOGIC$ (also denoted $\LOGIC$) is defined inductively as follows:
    \begin{enumerate}
        \item $\mathrm{A}(\LOGIC)\subseteq \LOGIC$;
        \item if $\phi_1,\ldots, \phi_n\in \LOGIC$ and $\heart\in \mathrm{C}(\LOGIC)$, and if $\heart[\phi_1,\ldots, \phi_n]$ is well-defined, then $\heart[\phi_1,\ldots, \phi_n]\in \LOGIC$.
    \end{enumerate}
    The set $\mathrm{Cxt}(\LOGIC)$ of \emph{contexts} of $\LOGIC$ is defined similarly, making use of a set $
    \{\cdot_n\mid n\in \mathbb{N}\}$ of designated atoms (i.e. placeholders):
    \begin{enumerate}
        \item $\mathrm{A}(\LOGIC)\cup\{\cdot_n\mid n\in \mathbb{N}\}\subseteq \mathrm{Cxt}(\LOGIC)$;
        \item if $\phi_1,\ldots, \phi_n\in \mathrm{Cxt}(\LOGIC)$ and $\heart\in \mathrm{C}(\LOGIC)$, and if $\heart[\phi_1,\ldots, \phi_n]$ is well-defined, then $\heart[\phi_1,\ldots, \phi_n]\in \mathrm{Cxt}(\LOGIC)$.
    \end{enumerate}
    Here we assume that whenever $\heart[\phi_1,\ldots, \phi_n]$ is well-defined for some formulas/designated atoms $\phi_1,\ldots,\phi_n$, then the result of replacing any $\phi_i$ in $\heart[\phi_1,\ldots, \phi_n]$ by a designated atom is also well-defined.
\end{definition}

We assume that all formulas are local (cf. Def.~\ref{def:locality}), and that each $n$-ary connective $\heart$ is \textit{compositional} in the following sense (cf. \cite{hodges2001compositionality,hodges2021meaningsyntax}): whenever $\varphi_i\equiv\varphi_i'$ for each $1\leq i\leq n$, then 
\begin{enumerate}
    \item[(i)]\label{enum:item(i)} $\heart[\varphi_1, \ldots, \varphi_n]$ is defined iff $\heart[\varphi_1', \ldots, \varphi_n']$ is defined; and
    \item[(ii)] if they are defined, then $\heart[\varphi_1, \ldots, \varphi_n]\equiv \heart[\varphi_1', \ldots, \varphi_n']$.
\end{enumerate}
(For more on these assumptions, consult the endnote\footnote{\label{c:footnote:meaning_assumptions}The locality and compositionality assumptions allow us to avoid tangential issues and to simplify our presentation. However, it should be noted that certain team-based logics are non-local (as already mentioned in Section \ref{c:section:closure_props}), and there are also some, including the first formulation of dependence logic \cite{vaananen2007} as well as BSML \cite{aloni2022}, that are not compositional in the sense of our assumption (ii) (there is, however, another sense in which they are compositional---see \cite{hodges1997,burgess,kontinen2011,anttila2024} for discussion). It should not be difficult to generalize our account to cover such cases, so we focus on the comparatively simple case of local and compositional logics.

Our assumption (i) ensures that the semantics $\llbracket\cdot\rrbracket$ mapping each well-formed $\phi$ to $\llbracket \phi \rrbracket$ is what Hodges \cite{hodges2001compositionality} calls \emph{Husserlian}: roughly, a semantics $\mu$ (a partial map from syntactic expressions to semantic values) is Husserlian if for any $\mu$-equivalent $\phi$ and $\psi$, any substitution result $\chi(\phi/\xi)$ has a semantic value according to $\mu$ ($\mu(\chi(\phi/\xi))$ is defined) iff $\chi(\psi/\xi)$ does. This assumption also simplifies the presentation, but, more crucially, it appears that a Husserlian semantics is what allows for an informative connection between the existence of compositional translations on the one hand and the definability of connectives on the other. 

As discussed in the main text, restricted-domain connectives such as the dependence connectives violate assumption (i), so we opt to change their formation rules. However, it is not difficult to see (a) that one can define a semantics that \emph{is} Husserlian for team logics incorporating such connectives and that preserves $\llbracket\cdot\rrbracket$-equivalences for intersubstitutable formulas; this semantics will satisfy its own versions of assumptions (i) and (ii), so (b) one can then develop the material in Sections \ref{subsec:definability}--\ref{subsec:compositionaltra} essentially as we have done in the main text, but with respect to this semantics. We return to this briefly in Appendix \ref{c:section:definability_results}.}.)

Observe that with their standard formation rules (see Definition \ref{c:def:syntax_PLNE}), the connectives $=_n\hspace{-0.1cm}(\cdot)$ fail assumption (i), since, e.g., $=_2\hspace{-0.1cm}(p,q)$ is defined but $=_2\hspace{-0.1cm}(p\land p,q)$ is not; similarly for the negation $\lnot$ of $\PLV$ and $\PLD$. Going forward, we will instead assume that the connectives $=_n\hspace{-0.1cm}(\cdot)$ are such that $=_n\hspace{-0.1cm}(\phi_1,\ldots,\phi_n)$ is defined whenever $\phi_1\equiv p_1,\ldots,\phi_n\equiv p_n$ for some $p_1,\ldots,p_n\in \mathsf{P}$, and similarly for $\lnot$, with semantics generalizing Definition \ref{c:definition:semantics_PLNE} in the obvious way.

For each of the logics $\LOGIC$ from Sections \ref{c:section:convex_prop} and \ref{c:section:union-closed_convex_prop}, we let $\mathrm{A}(\LOGIC)=\mathsf{P}$, and assume $\mathrm{C}(\LOGIC)$ consists of all other symbols of $\LOGIC$, e.g., 
\[
    \mathrm{C}(\DEP)=\{\bot,\top,\lnot, \land ,\lor\}\cup \{=_n\hspace{-0.1cm}(\cdot)\mid n\in \mathbb{N}\}.
\]
\begin{example}[Contexts depend on primitives/connectives classification]\label{c:example:funny_dependence_logic}
   Clearly dependence connectives $\dep{\cdot_1,\ldots}{\cdot_n}$ are contexts of $\DEP$. 
    
On the other hand, let  $\DEP'$ be the logic with $$\mathrm{A}(\DEP')=\mathsf{P}\cup \{\bot\}\cup \{\lnot p\mid p\in \mathsf{P}\}\cup\{\dep{p_1,\ldots}{p_n}\mid p_1,\ldots,p_n\in \mathsf{P} \}$$ and $\mathrm{C}(\DEP')=\{\land ,\lor\}$. This is the version of dependence logic considered in \cite{yang20173} (and discussed in endnote~\hyperref[endnote:dependence]{5}). It is easy to see that, in contrast, $\dep{\cdot_1,\ldots}{\cdot_n}$ are not contexts of $\DEP'$.
\end{example}

\subsection{Notions of Definability}\label{subsec:definability} We write $\mathsf{P}(\heart)$ for the set of propositional letters appearing in the connective/context $\heart$ (for connectives $\heart$, $\mathsf{P}(\heart)=\emptyset$). Because of our compositionality assumption, each $n$-ary connective or context $\heart$ (including formulas, thought of as nullary contexts) induces a partial function $\llbracket\heart\rrbracket:\mathbb{A}^n\to \mathbb{A}$, defined set-theoretically as the set of $n+1$-tuples
\begin{align*}
    \llbracket\heart\rrbracket:=\{(\llbracket\phi_1\rrbracket,\ldots,\llbracket\phi_n\rrbracket,\llbracket\heart(\phi_1,\ldots,\phi_n)\rrbracket)\in \mathbb{A}^{n+1} \mid&   \phi_1,\ldots,\phi_{n},\heart[\phi_1,\ldots,\phi_n]\\
    &\text{are well-defined}\}
\end{align*}
We treat $\llbracket\heart\rrbracket$ as representing the meaning of $\heart$.

For classes $\mathbb{P}_1,\ldots,\mathbb{P}_{n}\subseteq \mathbb{A}$ and $n$-ary $\heart$, we write $\heart^{\mathbb{P}_1\times \cdots \times\mathbb{P}_{n}}$ or $\heart^{\mathbb{P}_1, \ldots, \mathbb{P}_{n}}$ for the restricted partial function $\llbracket\heart\rrbracket{\upharpoonright}(\mathbb{P}_1\times \cdots \times\mathbb{P}_{n}\times \mathbb{A})\subseteq \mathbb{A}^{n+1}$; and we write simply $\heart^{\mathbb{P}}$ for $\heart^{{\mathbb{P}}^{n}}$. Note that for nullary $\heart$, $\heart^{\mathbb{P}}$ is just $\llbracket\heart\rrbracket=\{t\mid t\vDash \heart\}$.

 With these assumptions, a connective/context $\heart$ is uniformly defined by a connective/context $\medstar$ in a logic $\LOGIC$ in the sense of Definition \ref{c:definition:uniform_definability} if and only if $\heart^{\llbracket \LOGIC\rrbracket}=\medstar^{\llbracket \LOGIC\rrbracket}$ (in the simple case in which $\heart[\phi_1,\ldots,\phi_n]$ and $\medstar[\phi_1,\ldots,\phi_n]$ are always defined). This suggests the following generalization.

\begin{definition}[Generalized (uniform) definability]
An $n$-ary context $\heart$ is \emph{(uniformly) defined with respect to $\mathbb{P}_1\times\cdots\times\mathbb{P}_{n}\subseteq \mathbb{A}^n$} by an $n$-ary context $\medstar$ if $$\heart^{\mathbb{P}_1,\ldots,\mathbb{P}_{n}}\subseteq\medstar^{\mathbb{P}_1,\ldots,\mathbb{P}_{n}}.$$

$\heart$ is \emph{(uniformly) defined with respect to $\mathbb{P}\subseteq \mathbb{A}$} by $\medstar$ if it is defined with respect to $\mathbb{P}^n$ by $\medstar$. $\heart$ is \emph{globally (uniformly) defined} by $\medstar$ if it is defined by $\medstar$ with respect to $\mathbb{A}$; that is, if $\llbracket\heart\rrbracket\subseteq \llbracket\medstar\rrbracket$. If $\llbracket\heart\rrbracket=\llbracket\medstar\rrbracket$, we simply say that $\heart$ is \emph{defined by $\medstar$}.

$\heart$ is \emph{(uniformly) definable with respect to $\mathbb{P}_1,\ldots,\mathbb{P}_{n+1}$} in a logic
$\LOGIC$ if it is defined with respect to $\mathbb{P}_1,\ldots,\mathbb{P}_{n+1}$ by some context of $\LOGIC$, and similarly for the other definability notions. It is \emph{locally (uniformly) definable} in $\LOGIC$ if it is definable in $\LOGIC$ with respect to $\llbracket\LOGIC\rrbracket$.
\end{definition}

A few observations and examples are in order. First, observe that uniform definability as in Definition \ref{c:definition:uniform_definability} corresponds to local definability. Second, if $\heart[\phi_1,\ldots,\phi_n]$ is well-defined for all $\phi_1,\ldots,\phi_n$ with $\llbracket\phi_i\rrbracket\in \mathbb{P}_i$, then $\heart^{\mathbb{P}_1,\ldots,\mathbb{P}_{n}}\subseteq \medstar^{\mathbb{P}_1,\ldots,\mathbb{P}_{n}}$ implies $\heart^{\mathbb{P}_1,\ldots,\mathbb{P}_{n}}=\medstar^{\mathbb{P}_1,\ldots,\mathbb{P}_{n}}$. Third, for a nullary context/formula $\heart$, $\heart^\mathbb{P}\subseteq \medstar^\mathbb{P}$ iff $\llbracket\heart\rrbracket=\llbracket\medstar\rrbracket$ iff $\heart\equiv \medstar$. Fourth, $\heart$ defines $\medstar$ iff $\medstar$ defines $\heart$ iff [$\heart$ globally defines $\medstar$, and $\medstar$ globally defines $\heart$].

\begin{example} \label{c:example:definabilities}\
    \begin{itemize}
        \item[-]  $\vee^{\mathbb{D}}=\veedot^{\mathbb{D}}$, so $\vee$ is defined by $\veedot$ with respect to ${\mathbb{D}}$, and vice versa. Similarly $\vvee^{\mathbb{D}}=\vveedot^{\mathbb{D}}$. In contrast, by Facts \ref{c:fact:disjunctions_convexity_break} and \ref{c:fact:disjunctions_convexity_break2} and the fact that $\veedot$ and $\vveedot$ preserve convexity, $\vee^{\mathbb{C}}\neq  \veedot^{\mathbb{C}}$ and $\vvee^{\mathbb{C}}\neq\vveedot^{\mathbb{C}}$.
        \item[-] Since $\NE \equiv \filleddiamond \top$, the constant (nullary connective) $\NE$ is globally defined (and simply defined) by $\filleddiamond \top$, and vice versa.
        \item[-]  Recall the definition of the connectives $=_n\hspace{-0.1cm}(\cdot)$ in Example~\ref{c:example:funny_dependence_logic} and endnote~\hyperref[endnote:dependence]{5}. We have:
        \begin{align*}        
        \dep{\cdot_1,\ldots}{\cdot_n}^\mathbb{A}&={\dep{\cdot_1,\ldots}{\cdot_n}^{\mathbb{X}}}\\
        &= ((\cdot_1\vvee\lnot \cdot_1)\land \cdots \land (\cdot_{n-1}\vvee \lnot \cdot_{n-1}))\to (\cdot_n\vvee \lnot \cdot_n)^{\mathbb{X}}\\
        &\subseteq ((\cdot_1\vvee\lnot \cdot_1)\land \cdots \land (\cdot_{n-1}\vvee \lnot \cdot_{n-1}))\to (\cdot_n\vvee \lnot \cdot_n)^{\mathbb{A}},
        \end{align*}
        where $\mathbb{X}:=\{\llbracket p \rrbracket\mid p\in \mathsf{P}\}$. Therefore $=_n\hspace{-0.1cm}(\cdot)$ is globally defined by the context above.
        \item[-] By Proposition \ref{c:prop:global_disjunction_definable}, $(\filleddiamond \lnot \cdot_2 \to \cdot_1)^\mathbb
        F=(\cdot_1\vvee\cdot_2)^\mathbb{F}$.\footnote{Proposition~\ref{c:prop:global_disjunction_definable} shows that not only can one define a single application of $\vvee$ over flat properties $\mathbb{F}$, but also iterated applications:
            $$\left(\bigvveeds_{2\leq i\leq k}\cdot_i\right)^{\mathbb{F}}=\left(\left(\bigwedge_{2 \leq i\leq k}\filleddiamond \lnot\cdot_i\right) \to \cdot_1\right )^{\mathbb{F}}.$$
        This is strictly more general since the inquisitive disjunction does not preserve flatness. It might be of interest to define a notion of definability capturing such iterated applications.}
    \end{itemize}
\end{example}

\subsection{Compositional Extensions}\label{subsec:compositionalext} Given a set of primitives, connectives, and/or formulas $\mathrm{S}$ and class of properties $\mathbb{P}$, let $\llbracket S\rrbracket:=\{\llbracket\heart\rrbracket\mid \heart\in \mathrm{S}\}$ and $\mathrm{S}^\mathbb{P}:=\{\heart^\mathbb{P}\mid \heart\in \mathrm{S}\}$. Observe that for a logic $\LOGIC$, $\llbracket\mathrm{A}(\LOGIC)\rrbracket \subseteq\llbracket\LOGIC\rrbracket=\{\llbracket\heart\rrbracket\in \llbracket\mathrm{Cxt}(\LOGIC)\rrbracket \mid \heart$ is nullary$\}\subseteq \llbracket\mathrm{Cxt}(\LOGIC)\rrbracket$.

For sets $A,B$, we write $A\sqsubseteq B$ if for each $a\in A$, there is $b \in B$ such that $a\subseteq b$.

We are now ready to define our extension notions. The notions are presented in order of strength---we will prove the following proposition below:

\begin{proposition} \label{c:prop:extensions}
    Strict syntactic extension $\implies$ syntactic extension $\implies$ global compositional extension $\implies $ outer compositional extension $\implies $  compositional extension $\implies $  expressive extension.
\end{proposition}
The reader may wish to consult Example \ref{c:example:extensions}, below, as they read the following.

\begin{definition}[Compositional, syntactic, and expressive extensions] \label{c:def:extensions}
For $\LOGIC_1$ and $\LOGIC_2$ propositional team logics, we say that:
\begin{itemize}
    \item[-] \emph{Syntactic extensions}:
    \begin{itemize}
        \item[-] $\LOGIC_2$ is a \emph{strict syntactic extension} of $\LOGIC_1$ if each primitive/connective $\LOGIC_1$ is defined by a primitive/connective of $\LOGIC_2$: 
        $$ \text{$\llbracket\mathrm{A}(\LOGIC_1)\rrbracket\subseteq \llbracket\mathrm{A}(\LOGIC_2)\rrbracket$ and $\llbracket\mathrm{C}(\LOGIC_1)\rrbracket\subseteq \llbracket\mathrm{C}(\LOGIC_2)\rrbracket$.}$$
        \item[-] $\LOGIC_2$ is a \emph{syntactic extension} of $\LOGIC_1$ if each primitive of $\LOGIC_1$ is defined by a primitive of $\LOGIC_2$ and if each connective of $\LOGIC_1$ is globally defined by a connective of $\LOGIC_2$: 
        $$ \text{$\llbracket\mathrm{A}(\LOGIC_1)\rrbracket\subseteq \llbracket\mathrm{A}(\LOGIC_2)\rrbracket$ and $\llbracket\mathrm{C}(\LOGIC_1)\rrbracket\sqsubseteq \llbracket\mathrm{C}(\LOGIC_2)\rrbracket$.}$$
    \end{itemize}
    \item[-] \emph{Compositional extensions}:
    \begin{itemize}
        \item $\LOGIC_2$ is a \emph{global compositional extension} of $\LOGIC_1$ if each context of $\LOGIC_1$ is globally definable in $\LOGIC_2$: 
        $$ \text{$\llbracket\mathrm{Cxt}(\LOGIC_1)\rrbracket\sqsubseteq\llbracket\mathrm{Cxt}(\LOGIC_2)\rrbracket$.}$$
        \item $\LOGIC_2$ is an \emph{outer compositional extension} of $\LOGIC_1$ if each context of $\LOGIC_1$ is locally definable in $\LOGIC_2$: 
        $$ \text{$\mathrm{Cxt}(\LOGIC_1)^{\llbracket\LOGIC_2\rrbracket}\sqsubseteq\mathrm{Cxt}(\LOGIC_2)^{\llbracket\LOGIC_2\rrbracket}$.}$$
        \item $\LOGIC_2$ is a \emph{compositional extension} of $\LOGIC_1$ if each context of $\LOGIC_1$ is definable in $\LOGIC_2$ with respect to $\llbracket\LOGIC_1\rrbracket$: 
        $$ \text{$\mathrm{Cxt}(\LOGIC_1)^{\llbracket\LOGIC_1\rrbracket}\sqsubseteq\mathrm{Cxt}(\LOGIC_2)^{\llbracket\LOGIC_1\rrbracket}$.}$$
    \end{itemize}
    \item[-] \emph{Expressive extension}: $\LOGIC_2$ is an \emph{expressive extension} of $\LOGIC_1$ if 
     $$ \text{$\llbracket\LOGIC_1\rrbracket\subseteq \llbracket\LOGIC_2\rrbracket$.}$$
\end{itemize}
\end{definition}

Note that $\LOGIC_2$ being a global compositional extension of $\LOGIC_1$ is equivalent to the following: each primitive of $\LOGIC_1$ is expressible in $\LOGIC_2$ and each connective of $\LOGIC_1$ is globally definable in $\LOGIC_2$, i.e., $\llbracket\mathrm{A}(\LOGIC_1)\rrbracket\subseteq \llbracket\LOGIC_2\rrbracket$ and $\llbracket\mathrm{C}(\LOGIC_1)\rrbracket\sqsubseteq\llbracket\mathrm{Cxt}(\LOGIC_2)\rrbracket$. Similarly, e.g., $\LOGIC_2$ being a compositional extension of $\LOGIC_1$ is equivalent to: $\llbracket\mathrm{A}(\LOGIC_1)\rrbracket\subseteq \llbracket\LOGIC\rrbracket$ and $\mathrm{C}(\LOGIC_1)^{\llbracket\LOGIC_1\rrbracket}\sqsubseteq\mathrm{Cxt}(\LOGIC_2)^{\llbracket\LOGIC_1\rrbracket}$. (Strictly speaking, the notion of expressive extension above is weaker than one derived from the notion of expressive completeness from Definition \ref{c:definition:expressive_completeness}, since, e.g., $\llbracket p\vee\lnot p\rrbracket=\llbracket \top\rrbracket$, even if $\mathsf{P}(p\vee\lnot p)=\{p\}\neq \varnothing= \mathsf{P}(\top)$.)

\begin{proof}[Proof of Proposition \ref{c:prop:extensions}] We show that outer compositional extension $\implies$ compositional extension; the other implications are either obvious, or have similar proofs. Let $\LOGIC_2$ be an outer compositional extension of $\LOGIC_1$. It suffices to show that $\llbracket\LOGIC_1\rrbracket\subseteq \llbracket\LOGIC_2\rrbracket$, i.e. $\LOGIC_2$ is an expressive extension of $\LOGIC_1$, for then $\mathrm{Cxt}(\LOGIC_1)^{\llbracket\LOGIC_2\rrbracket}\sqsubseteq\mathrm{Cxt}(\LOGIC_2)^{\llbracket\LOGIC_2\rrbracket}$ implies  $\mathrm{Cxt}(\LOGIC_1)^{\llbracket\LOGIC_1\rrbracket}\sqsubseteq\mathrm{Cxt}(\LOGIC_2)^{\llbracket\LOGIC_1\rrbracket}$. We show by induction on $\phi\in \LOGIC_1$ that $\llbracket\phi \rrbracket\in\llbracket\LOGIC_2\rrbracket$. If $\phi\in \mathrm{A}(\LOGIC_1)$, then as $\mathrm{Cxt}(\LOGIC_1)^{\llbracket\LOGIC_2\rrbracket}\sqsubseteq\mathrm{Cxt}(\LOGIC_2)^{\llbracket\LOGIC_2\rrbracket}$ and $\llbracket\phi \rrbracket\in\llbracket\mathrm{A}(\LOGIC_2)\rrbracket\subseteq \mathrm{Cxt}(\LOGIC_1)^{\llbracket\LOGIC_2\rrbracket}$, we have $\llbracket\phi \rrbracket\in \llbracket\LOGIC_2\rrbracket$. Now suppose $\phi=\heart[\psi_1,\ldots,\psi_n]$. By the induction hypothesis we have $\{\llbracket\psi_1\rrbracket,\ldots,\llbracket\psi_n\rrbracket\}\subseteq \llbracket\LOGIC_2\rrbracket$, and by $\mathrm{Cxt}(\LOGIC_1)^{\llbracket\LOGIC_2\rrbracket}\sqsubseteq\mathrm{Cxt}(\LOGIC_2)^{\llbracket\LOGIC_2\rrbracket}$, there is some $\medstar \in \mathrm{Cxt}(\LOGIC_2)$ such that $\heart^{\llbracket\LOGIC_2\rrbracket}\subseteq \medstar^{\llbracket\LOGIC_2\rrbracket}$. But then $(\llbracket\psi_1\rrbracket,\ldots,\llbracket\psi_n\rrbracket,\llbracket\phi\rrbracket)\in \heart^{\llbracket\LOGIC_2\rrbracket}\subseteq \medstar^{\llbracket\LOGIC_2\rrbracket}$, whence $\llbracket\phi \rrbracket\in \llbracket \LOGIC_2\rrbracket$.
\end{proof}

\begin{example}\label{c:example:extensions}\
\begin{enumerate}
    \item[(a)] \emph{Strict syntactic}: $\CONDEP$ is a strict syntactic extension of $\PLD$, but not of $\PLV$.
    \item[(b)] \emph{Syntactic}: Let the connective $=^E_n\hspace{-0.1cm}(\cdot)$ denote the connective corresponding to an $n$-ary \emph{extended dependence atom} (see, e.g., \cite{ebbing,hella2014}). We assume (departing from the literature in order to satisfy assumption (i) from Section \ref{section:set-up}) that $=^E_n\hspace{-0.1cm}(\phi_1,\ldots,\phi_n)$ is well-defined for all flat $\phi_1,\ldots,\phi_n$, yielding the formula $=^E\hspace{-0.1cm}(\phi_1,\ldots,\phi_n)$ with the semantics:
    \begin{align*}
        t\vDash {=^E\hspace{-0.1cm}(\phi_1,\ldots,\phi_n)}:\iff &\text{for all }v,w\in t:\text{ if }[\forall 1\leq i\leq n-1:\{v\}\vDash \phi_i \iff \{w\}\vDash \phi_i ],\\
    &\text{then } [\{v\}\vDash \phi_n \iff \{w\}\vDash \phi_n].
    \end{align*}
    Let $\PLV(=^E\hspace{-0.1cm}(\cdot))$ be the logic obtained by replacing each $=_n\hspace{-0.1cm}(\cdot)$ with $=^E_n\hspace{-0.1cm}(\cdot)$ in $\DEP$. Since $\llbracket=_n\hspace{-0.1cm}(\cdot)\rrbracket\subset \llbracket=^E_n\hspace{-0.1cm}(\cdot)\rrbracket$, $\PLV(=^E\hspace{-0.1cm}(\cdot))$ is a non-strict syntactic extension of $\DEP$.
    \item[(c)] \emph{Global compositional}: Since $\NE\equiv\filleddiamond\top$, $\PLIM$ is a global compositional extension (but not a syntactic extension) of $\PLI(\NE)$---$\PLI$ with the nullary connective $\NE$.
    \item[(d)] \emph{Outer compositional}: Let $\LOGIC_1=(\mathsf{P},\{\vee\})$ and $\LOGIC_2=(\mathsf{P},\{\veedot\})$. Then clearly $\llbracket\LOGIC_1\rrbracket=\llbracket\LOGIC_2\rrbracket\subseteq \mathbb{D}$, whence $\vee^{\llbracket\LOGIC_2\rrbracket}=\veedot^{\llbracket\LOGIC_2\rrbracket}$, and so $\LOGIC_2$ is an outer compositional extension of $\LOGIC_1$. It is easy to see that for all binary contexts $\heart$ in $\LOGIC_2$ and all formulas $\phi,\psi$, $\heart[\phi,\psi] $ is downward closed, so $\vee^{\mathbb{A}}\neq\heart^{\mathbb{A}}$. Therefore, $\LOGIC_2$ is not a global compositional extension of $\LOGIC_1$.
    \item[(e)] \emph{Compositional}: Since $\vee^{\mathbb{DE}}=\veedot^{\mathbb{DE}}$ and $\vvee^{\mathbb{DE}}=\vveedot^{\mathbb{DE}}$, we have that $\vee$ is definable in $\CONDEP$ with respect to $\mathbb{DE}=\llbracket \DEP\rrbracket$, and $\vvee$ is definable in $\CONINQ$ with respect to $\mathbb{DE}=\llbracket \INQ\rrbracket$. It is then easy to see that $\CONDEP$ is a compositional extension of $\DEP$, and similarly with $\CONINQ$ and $\INQ$. Given Fact \ref{c:fact:disjunctions_convexity_break2}, however, $\CONDEP$ is not an outer compositional extension of $\DEP$, and $\CONINQ$ is not an outer compositional extension of $\INQ$.
    \item[(f)] \emph{Expressive}: $\INQ$ is an expressive extension of $\DEP$ (and vice versa), but not of $\CONINQ$. We mentioned above that $\vee$ is not (locally) definable in $\INQ$. It is therefore not definable in $\INQ$ with respect to $\llbracket \INQ\rrbracket=\llbracket \DEP\rrbracket$; therefore, $\INQ$ is not a compositional extension of $\DEP$.
    \item[(g)] $\DEP$ is a global compositional extension of $\DEP'$ (see Example \ref{c:example:funny_dependence_logic}), but $\DEP'$ is only an expressive extension of $\DEP$.
    \item[(h)] $(\DEP,\emptyset)$ is a (strict) syntactic extension of $(\INQ,\emptyset)$ (and vice versa).
\end{enumerate}
\end{example}
With this, we have precisely characterized what we mean by calling $\CONDEP$ a convex variant of dependence logic $\DEP$, and $\CONINQ$ a convex variant of inquisitive logic $\INQ$: the former are compositional extensions of the latter. 

By contrast, $\PLIM$  is not a compositional extension of $\INQ$. This follows from recent unpublished work \cite{barbero2026convex}, written in response to an earlier version of this paper, where Fausto Barbero has shown that $\vvee$ is not definable in $\PLIM$ with respect to $\mathbb{DE}$.

 (What about convex compositional extensions of the convex union-closed logic $\PLNE$ discussed in Section \ref{c:section:union-closed_convex_prop}? It is not difficult to see that $\mathrm{PL}_{\vee_C}(\con{\cdot},\filleddiamond)$, i.e., the result of replacing $\veedot$ in $\CONDEP$ with the disjunction $\vee_C$ considered in endnote \hyperref[endnote:convex_closure_disjunctions]{3}, is a compositional extension of both $\DEP$ and $\PLNE$.)

Example \ref{c:example:extensions} (g) and (h) demonstrate how the extension notions depend on how the symbols of the logic are partitioned into primitives and connectives. Our framework places no restrictions on how this partitioning is carried out---this is to allow us to account for definability results in the literature which treat dependence atoms and negated formulas as primitives \cite{yang20173,barbero2024}, but it also leads to a very permissive notion of what can count as a logic, as evinced by (h). The partioning may be viewed as a conceptual tool: intuitively, the symbols placed in $\mathrm{A}(\LOGIC)$ are those whose only contribution to meaning in $\LOGIC$ is taken to be the properties they express, whereas the symbols placed in $\mathrm{C}(\LOGIC)$ are those whose nature as connectives is also considered essential to meaning in $\LOGIC$, and which ought to be captured in a faithful way in a compositional translation of the logic, as discussed in the next section.

See Appendix \ref{c:section:definability_results} for some further definability and extension facts.

\subsection{Compositional Translations}\label{subsec:compositionaltra} We now establish the link between compositional translations and extensions. For convenience we recall the definition of the former.

\setcounter{theorem}{3} 
\begin{definition}[Restated: Compositional translation]
      A map $\tau:\LOGIC_1\to \LOGIC_2$ is a \emph{compositional translation} or \emph{translational embedding} if it satisfies the following two conditions. 
    \begin{enumerate}
        \item $\tau$ is \emph{meaning-preserving}: for all formulas $\varphi$ of $\LOGIC_1$, $ \varphi \equiv \tau(\varphi) $. 
        \item $\tau$ is \emph{compositional}: for each connective $\heart$ of $\LOGIC_1$, there is a context $\theta_{\heart}$ of $\LOGIC_2$ such that $\tau(\heart[\varphi_1, \ldots, \varphi_n])=\theta_{\heart}[\tau(\varphi_1), \ldots, \tau(\varphi_n)]$ for all $\varphi_1, \ldots, \varphi_n$ of $\LOGIC_1$.
    \end{enumerate}
\end{definition}
\setcounter{theorem}{11} 

\begin{proposition}
    $\LOGIC_2$ is a compositional extension of $\LOGIC_1$ iff there is a compositional translation $\tau:\LOGIC_1\to \LOGIC_2$.
\end{proposition}
\begin{proof}
$\Longrightarrow$: By $\mathrm{Cxt}(\LOGIC_1)^{\llbracket\LOGIC_1\rrbracket}\sqsubseteq\mathrm{Cxt}(\LOGIC_2)^{\llbracket\LOGIC_1\rrbracket}$, we find, for each primitive $\phi\in \mathrm{A}(\LOGIC_1)$, a formula $\psi_\phi\in \LOGIC_2$ such that $\llbracket\phi\rrbracket= \llbracket\psi_{\phi}\rrbracket$, and for each connective $\heart\in \mathrm{C}(\LOGIC_1)$, a context $\theta_{\heart}\in \mathrm{Cxt}(\LOGIC_2)$ such that $\heart^{\llbracket\LOGIC_1\rrbracket}\subseteq \theta_{\heart}^{\llbracket\LOGIC_1\rrbracket}$. Define $\tau:\LOGIC_1\to \LOGIC_2$ as follows:
    \begin{align*}
        \tau(\varphi) \mathrel{:=} 
        \begin{cases}
            \psi_\phi, & \text{if } \varphi\in \mathrm{A}(\LOGIC_1); \\
            \theta_{\heart}[\tau(\varphi_1), \ldots, \tau(\varphi_n)] , & \text{if } \varphi =\heart[\varphi_1, \ldots, \varphi_n].
        \end{cases}
    \end{align*}
    Then $\tau$ is compositional, and an induction on the complexity of formulas shows that it is meaning-preserving.
    
    $\Longleftarrow$: It suffices to show that $\llbracket\mathrm{A}(\LOGIC_1)\rrbracket\subseteq \llbracket\LOGIC_2\rrbracket$ and $\mathrm{C}(\LOGIC_1)^{\llbracket\LOGIC_1\rrbracket}\sqsubseteq\mathrm{Cxt}(\LOGIC_2)^{\llbracket\LOGIC_1\rrbracket}$. For the former, by meaning-preservation we have $\llbracket\phi \rrbracket=\llbracket \tau(\phi)\rrbracket \in \llbracket \LOGIC_2\rrbracket$  for all $\phi \in \mathrm{A}(\LOGIC_1)$. As for the latter, for $n$-ary $\heart\in \mathrm{C}(\LOGIC_1)$, take $\theta_{\heart}$ given by the compositionality requirement. We show that 
    \[
    \heart^{\llbracket \LOGIC_1 \rrbracket}\subseteq \theta_{\heart}^{\llbracket \LOGIC_1 \rrbracket}.
    \]
    For this, let $ \phi_1, \ldots, \phi_n\in\LOGIC_1$ be such that $\heart[\phi_1,\ldots,\phi_n]$ is well-defined. By meaning-preservation, we have $\llbracket \phi_i\rrbracket=\llbracket\tau( \phi_i)\rrbracket$ for $1\leq i \leq n$, and by the compositionality requirement, $\theta_{\heart}[ \tau(\phi_1), \ldots,  \tau(\phi_n)]$ is well-defined. Thus, it suffices to verify that $$\llbracket\heart\rrbracket(\llbracket \phi_1\rrbracket, \ldots, \llbracket \phi_n\rrbracket)=  \llbracket\theta_{\heart} \rrbracket(\llbracket\tau( \phi_1)\rrbracket, \ldots, \llbracket\tau( \phi_n)\rrbracket),$$ and this immediately follows by meaning-preservation and compositionality.
\end{proof}

One may formulate similar translation notions corresponding to each of our extension notions. We give two such correspondences in the proposition below, which we state without proof.

\begin{proposition}\
    \begin{itemize}
            \item[-] $\LOGIC_2$ is a expressive extension of $\LOGIC_1$ iff there is a translation from $\LOGIC_1$ to $\LOGIC_2$---a meaning-preserving map $\tau:\LOGIC_1\to \LOGIC_2$.
            \item[-]     $\LOGIC_2$ is an outer compositional extension of $\LOGIC_1$ iff there is a compositional translation $\tau:\LOGIC_1\to \LOGIC_2$ that is additionally such that for each $\heart\in \mathrm{C}(\LOGIC_1)$, the context $\theta_{\heart}$ of $\LOGIC_2$ given by the compositionality condition locally defines $\heart$ in $\LOGIC_2$ (this seems to be the notion of compositional translation discussed in \cite{yang20173}).
    \end{itemize}
\end{proposition}
\section{Conclusion}\label{c:section:conclusion}

In this paper, we introduced the logics $\CONDEP$, $\CONINQ$ and $\PLIM$ and proved them to be expressively complete for all convex propositional properties. We also introduced modal extensions and showed analogous modal expressive completeness results. We then examined some union-closed convex logics from the literature: we showed that the propositional logic $\PLNE$ is expressively complete for the class of all union-closed convex propositional properties, and we proved modal analogues of this result for the modal extensions $\ML_{\vee,\Diamond}(\NE)$, $\ML_{\vee,\diamonddiamond}(\NE)$ and $\BSML$.

We also showed that while the split disjunction $\vee$ and the diamond $\diamonddiamond$ preserve convexity in a convex and union-closed setting, they do not preserve convexity in general. Similarly, while the inquisitive disjunction $\vvee$ preserves downward closure (and hence convexity) in a downward-closed setting, it does not preserve convexity in general. 

Finally, we generalized the notion of uniform definability in order to introduce the notion of compositional extension, together with other extension notions. We showed that convex dependence logic $\CONDEP$ is a compositional extension of its downward-closed counterpart, viz. dependence logic $\DEP$, and similarly for the pair of convex inquisitive logic $\CONINQ$ and inquisitive logic $\INQ$. Last, we proved that a logic is a compositional extension of another just in case there is a compositional translation between the two.

We note three more problems for future research to add to those mentioned in Sections~\ref{c:section:convex_modal} (see also Appendix~\ref{c:section:definability_results}):

\textit{Problem(s) 1: Axiomatizations of convex team logics.} There is a commonly-used strategy for finding axiomatizations of propositional and modal team logics (see, for instance, \cite{CiardelliRoelofsen2011,yangvaananen2016,yang2017,yang20172,yang2022}) that makes heavy use of the characteristic formulas for properties provided by expressive completeness results such as those established in this paper. We expect it to be possible to use this strategy to construct axiomatizations of the logics we have studied here as well. It would be particularly interesting to see axiomatizations of the convex propositional logics, to see how these axiomatizations differ from those of the downward-closed logics $\INQ$ and $\DEP$. (Note that $\PLNE$ has already been axiomatized in \cite{yang2017}, and $\ML_{\Diamond,\vee}(\NE)$ and $\BSML$ in \cite{aloni2023}. An extension of $\ML_{\diamonddiamond,\vee}(\NE)$ has been axiomatized in \cite{anttila2021}.)

\textit{Problem(s) 2: Convex predicate team logics.} First-order dependence logic (D) coincides in expressive power with the downward-monotone (or downward-closed) fragment of existential second-order logic (ESO) in the sense that a class of nonempty teams is definable by a formula of D just in case it is definable by a sentence of this fragment \cite{kontinen2009}. Similarly, we have, for instance, that first-order \emph{independence logic} coincides with full ESO \cite{galliani2012} in this sense; D with the \emph{Boolean negation} coincides with full second-order logic (SO) \cite{nurmi}; D with $\to$ coincides with the downward-closed fragment of SO \cite{yang2013,yang2014}; and there are other team logics coinciding in this way with the union-closed fragment of ESO \cite{hoelzel2021}, as well as with both full first-order logic and its downward-closed fragment \cite{kontinen2023}. Is it possible to construct a convex variant of first-order dependence logic (or indeed any natural team logic, perhaps dissimilar to dependence logic) which is expressively equivalent with the convex fragment of ESO? Similarly, can we find a natural team logic equivalent to the convex union-closed fragment of ESO? How about the convex (and convex and union-closed) fragments of first-order and second-order logic?

\textit{Problem(s) 3: Expressive completeness for contexts (``functional completeness'').} Can we generalize the expressive completeness theorem for a given team logic $\LOGIC$ to some natural characterization of $\mathrm{Cxt}(\LOGIC)^{\llbracket \LOGIC \rrbracket}$ (or $\mathrm{Cxt}(\LOGIC)^{\mathbb{A}}$)? (Cf. Post's functional completeness theorem---see, for instance, \cite{post,pelletiernorman1990functionalcompleteness}.)
\appendix

\section{Addendum to Section~\ref{c:section:uniform_definability}} \label{c:section:definability_results} 
We list, for ease of reference, some of the (local) undefinability and related results for propositional team logics from the literature in our context-sensitive notation of Sections~\ref{section:set-up}--\ref{subsec:compositionaltra}, as well as some additional results and open problems. We also note that certain results employing a restricted notion of context in the literature generalize to results with a more liberal notion of context. 

\begin{remark} \label{c:remark:definability_results} We list results to the effect that a connective $\heart$ is not locally definable in a logic $\LOGIC$, or to the effect that $\heart$ is not \emph{eliminable} from $\LOGIC$: removing $\heart$ from $\LOGIC$ results in an expressively weaker logic $\LOGIC'$ ($\LOGIC'$ is not an expressive extension of $\LOGIC$)---\emph{a fortiori}, $\heart$ is not locally definable in $\LOGIC'$. 
\begin{itemize}
    \item[\cite{yang20173}]: $\vvee$ and $\to$ are not locally definable in $\DEP'$ (see Example \ref{c:example:funny_dependence_logic}).
    \item[\cite{ciardelli2019}]: The set of connectives of $\PLI(\vvee,\vee)$, i.e., the logic $(\mathsf{P},\{\bot,\land,\to,\vvee,\vee\})$ is {independent} in that each connective is either not eliminable or not locally definable in the logic given by removing that connective from $\PLI(\vvee,\vee)$: $\bot,\vvee,\to$ are not eliminable; $\vee,\land$ are not definable in this sense.
    
    \item[\cite{barbero2024}]: Classify the syntax of inquisitive logic as $(\mathsf{P},\{\bot,\top,\land,\vvee,\vee,\lnot_i,?\})$, where $\lnot_i$ is unary with $\lnot_i\phi\equiv (\phi \to \bot)$ and $?$ is unary with $? \phi\equiv \phi \vvee \neg_i\phi$. It is shown that $\land,\vee,\to$ are independent from the other connectives in that each  is not locally definable in the logic given by removing that connective.
    
    It is also shown that $\vvee$ is not locally definable in the extended-atom version of $\DEP'$, i.e., in the logic with the primitives:  $$\mathsf{P}\cup \{\bot\}\cup \{\lnot p\mid p\in \mathsf{P}\}\cup\{=^E\hspace{-0.1cm}(\alpha_1,\ldots,\alpha_n)\mid \alpha_1,\ldots,\alpha_n\in \PLV \},$$
    and the connectives $\land,\lor$.
\end{itemize}
In addition, in response to open problems listed in an earlier version of this paper, Fausto Barbero has proved (in an unpublished manuscript):
\begin{itemize}
    \item[\cite{barbero2026convex}]: Neither $\vvee$ nor $\vveedot$ is definable in $\PLIM$ with respect to $\mathbb{DE}$---this implies, in particular, that $\PLIM$ is not a compositional extension of $\INQ$ or of $\CONINQ$.
    
    $\vveedot$ is not locally definable in the variant $\CONDEP'$ of convex dependence logic with the primitives $$\mathsf{P}\cup\{\bot\}\cup \{\lnot \alpha\mid \alpha \in \PLD\}\cup\{\dep{p_1,\ldots}{p_n}\mid p_1,\ldots,p_n\in \mathsf{P} \},$$
where $t\vDash \lnot \alpha \iff \forall v\in t:\{v\}\nvDash \alpha$; and the connectives $\land,\veedot$. This implies that $\CONDEP'$ is not a compositional extension of $\CONINQ$.
\end{itemize}
\end{remark}

 For the results in Remark \ref{c:remark:definability_results} relating to logics that classify symbols other than those in $\mathsf{P}$ as primitives, it is natural to ask whether one can prove the analogous result for the corresponding logic in which the set of primitives is simply $\mathsf{P}$. E.g., are $\to$ and $\vvee$ locally definable in $\DEP$? If we consider the logics as they appear in the literature---that is, with their original formation rules, dropping assumption our assumption (i) from Section \ref{section:set-up}---the answer is positive, as shown in the following fact. For disambiguation, we use underline, as in $\underline{=_n\hspace{-0.1cm}(\cdot)}$ and $\underline{\lnot}$, to refer to the connectives with their original formation rules. (We assume for convenience that we have a Husserlian semantics $\llbracket \cdot\rrbracket$ for these logics, and that we have extended the machinery of Sections \ref{subsec:definability} and \ref{subsec:compositionalext} to this semantics (see endnote \hyperref[c:footnote:meaning_assumptions]{6}).)
\begin{fact} \label{c:fact:transfer_of_undefinability_results}
    Each local undefinability result in Remark \ref{c:remark:definability_results} relating to a logic $\LOGIC'$ with $\mathrm{A}(\LOGIC')\neq \mathsf{P}$ implies the corresponding local undefinability result for the logic $\LOGIC$ given by classifying the symbols in $\mathrm{A}(\LOGIC')\setminus \mathsf{P}$ as connectives, with formation rules restricting the connectives $\underline{=_n\hspace{-0.1cm}({\cdot})}$, $\underline{=^E_n\hspace{-0.1cm}(\cdot)}$, and $\underline{\lnot}$ to propositional letters or classical formulas, as appropriate.
\end{fact}
\begin{proof}[Proof sketch]
    We show that the local undefinability of $\vvee$ in $\DEP'$ implies the local undefinability of $\vvee$ in $\PLV(\underline{\con{\cdot}})$; the proofs for the other results are similar.   For contraposition, assume that $\heart[\cdot_1,\cdot_2]\in \mathrm{Cxt}(\PLV(\underline{\con{\cdot}}))$ locally defines $\vvee$ in $\PLV(\underline{\con{\cdot}})$. Then, in particular,  $\llbracket\heart \rrbracket(\llbracket \underline{\con{p}}\rrbracket,\llbracket\underline{ \con{q}}\rrbracket)=\llbracket\vvee\rrbracket(\llbracket \underline{\con{p}}\rrbracket,\llbracket\underline{\con{q}}\rrbracket)$, so $\heart[\phi,\psi]$ must be well-defined for all $\phi,\psi$ with $\llbracket\phi\rrbracket=\llbracket\underline{\con{p}}\rrbracket$ and $\llbracket\psi\rrbracket=\llbracket\underline{\con{q}}\rrbracket$; in particular, $\heart[\underline{\con{p}},\underline{\con{q}}]$ must be well-defined. But this means the designated atoms $\cdot_1$, $\cdot_2$ cannot occur in the scope of negation or functional dependence connectives in $\heart$. It is then easy to see that $\heart$ is also a context of $\DEP'$.
\end{proof}

 Using this, we can derive several further compositional extension facts concerning the convex logics with their original formation rules. Given Fact \ref{c:fact:transfer_of_undefinability_results}, the results in \cite{barbero2026convex} listed in Remark \ref{c:remark:definability_results}  show that $\PLD(\underline{\con{\cdot}},\filleddiamond)$ is not a compositional extension of $\CONINQ$. We also note, without going into details, that essentially the same argument as that given for the local undefinability of $\to$ in $\DEP'$ in \cite{yang20173} also shows that $\to$ is locally undefinable in $\CONDEP'$, and that Fact \ref{c:fact:transfer_of_undefinability_results} can be extended to this result, whence $\to$ is not locally definable in $\PLD(\underline{\con{\cdot}},\filleddiamond)$ and $\PLD(\underline{\con{\cdot}},\filleddiamond)$ is not a compositional extension of $\PLIM$. Whether the analogous results hold for the variants of these logics that do abide by assumption (i) in Section \ref{section:set-up} is an open problem. We list one further problem:

\textit{Problem(s) 4:  Local definability of $\veedot$.} Is $\veedot$ locally definable in $\PLIM$ or in $\CONINQ$? (In other words, is $\PLIM$/$\CONINQ$ a compositional extension of $\CONDEP$? Note that it follows from facts listed in Example \ref{c:example:definabilities} that the connectives $=_n\hspace{-0.1cm}(\cdot)$ are globally definable in both  $\CONINQ$ and $\PLIM$.) (It might be possible to adapt the proof from \cite{ciardelli2019} showing that $\vee$ is not locally definable in $\INQ$ to show that this is not the case.)

Note that together with the results listed above, the local undefinability of $\veedot$ in $\CONINQ$ would imply that the convex variants $\PLD(\underline{\con{\cdot}},\filleddiamond)$ and $\CONINQ$ are related to one another in the same manner as are propositional dependence logic $\DEP$ and propositional inquisitive logic $\PLIM$: expressively equivalent, but neither a compositional extension of the other.

\bibliography{bibb}%

 \begin{acks}
 A preliminary version of this paper appeared in Anttila's PhD dissertation \cite{anttila2025}; the authors thank Anttila's supervisors, Maria Aloni and Fan Yang, for extensive feedback. We are also especially grateful to Fausto Barbero for his close engagement with earlier drafts of the paper, which led to several improvements. We thank an anonymous referee for their helpful comments, and Rodrigo Nicolau Almeida, Nick Bezhanishvili, Ivano Ciardelli, Josef von-Hoffmann Doyle, Pietro Galliani, Lorenz Hornung, Juha Kontinen, Aybüke Özgün, Vít Punčochář, Marius Tritschler, and Jouko Väänänen for discussions on the content of the paper.\\
 
 Anttila received funding from grants 336283 and 368671 of the Research Council of Finland as well as the European Research Council (ERC) grants TEAMDEP (101020762) and InqML (101116774). Knudstorp was supported by Nothing is Logical (NihiL), an NWO OC project (406.21.CTW.023).\\

This project has received funding from the European Research Council (ERC) under the European Union's Horizon 2020 research and innovation programme (grant agreement n° 101020762).\\

Funded by the European Union. Views and opinions expressed are however those of the author(s) only and do not necessarily reflect those of the European Union or the European Research Council Executive Agency. Neither the European Union nor the granting authority can be held responsible for them. This work is supported by ERC grant InqML (grant agreement n° 101116774).
 \end{acks}

\end{document}